\documentclass{amsart}
\usepackage[margin=1.2in]{geometry}
\usepackage[T1]{fontenc}
\usepackage[utf8]{inputenc}
\usepackage{lmodern}
\usepackage{amsmath,amssymb,amsfonts,amsthm,amscd,mathtools}
\usepackage{latexsym,mathrsfs,stmaryrd}
\usepackage[dvipsnames]{xcolor}
\usepackage[all]{xy}
\usepackage{verbatim}
\usepackage{tikz-cd}
\usepackage{booktabs,array}
\usepackage{hyperref}
\hypersetup{
  pdftitle={Twisting of affine algebraic groups, III},
  pdfauthor={Shlomo Gelaki}
}

\hypersetup{
    colorlinks=true,
    linkcolor=Blue,
    urlcolor=Blue,
    citecolor=Blue,
    filecolor=Blue,
    pdfborder={0 0 0}
}

\newtheorem{theorem}{Theorem}[section]
\newtheorem{lemma}[theorem]{Lemma}
\newtheorem{proposition}[theorem]{Proposition}
\newtheorem{corollary}[theorem]{Corollary}
\theoremstyle{definition}
\newtheorem{definition}[theorem]{Definition}

\newtheorem{example}[theorem]{Example}
\newtheorem{question}[theorem]{Question}
\newtheorem{conjecture}[theorem]{Conjecture}
\newtheorem{remark}[theorem]{Remark}

\numberwithin{equation}{section}

\newcommand{\Rep}{\operatorname{Rep}}

\newcommand{\Hom}{\operatorname{Hom}}
\newcommand{\End}{\operatorname{End}}

\newcommand{\Rad}{\operatorname{Rad}}

\newcommand{\Os}{\mathscr{O}}

\newcommand{\id}{\mathrm{id}}

\newcommand{\D}{\mathcal D}

\newcommand{\C}{\mathbb C}
\newcommand{\Ga}{\mathbb G_a}
\newcommand{\Gm}{\mathbb G_m}

\newcommand{\eps}{\varepsilon}

\newcommand{\Lie}{\operatorname{Lie}}
\newcommand{\Supp}{\operatorname{Supp}}
\newcommand{\Aut}{\operatorname{Aut}}
\newcommand{\GKdim}{\operatorname{GKdim}}
\newcommand{\CYB}{\operatorname{CYB}}
\newcommand{\Jac}{\operatorname{Jac}}

\newcommand{\ad}{\operatorname{ad}}
\newcommand{\im}{\operatorname{im}}
\newcommand{\res}{\operatorname{res}}
\newcommand{\Jset}{\mathcal J}

\newcommand{\Jcs}{\mathcal J_{\mathrm{cs}}}

\title[Twisting of affine algebraic groups III]{Twisting of affine algebraic groups, III}
\author{Shlomo Gelaki}
\address{Department of Mathematics, Iowa State University, Ames, IA 50011, USA}
\email{gelaki@iastate.edu}
\date{\today}


\begin{document}

\begin{abstract}
Let $G$ be an affine algebraic group over $\C$ and let $J$ be a Hopf
$2$-cocycle for $\Os(G)$. We prove that one-sided cocycle twisting
preserves Gelfand--Kirillov dimension and that $\Os(G)_J$ is simple
if and only if $J$ is minimal. If $H$ is the support of $J$, we identify
the primitive and maximal spectra of $\Os(G)_J$ with the closed points of $G/H$ and
show that all primitive quotients are isomorphic to the minimal support
algebra $\Os(H)_J$. In particular, the prime and Jacobson radicals vanish, and
finite-dimensional irreducible modules exist exactly when $H$ is finite.
For connected solvable $G$, we construct an iterated differential Ore
presentation over an algebraic quantum torus and deduce that $\Os(G)_J$
is an affine Noetherian domain. We describe the torus-trivial Hopf $2$-cocycles on $\Os(G)$
using nilpotent solutions of the classical Yang--Baxter equation and
identify their supports with the algebraic hulls of their carriers.
We also separate the centrally supported toral sector and formulate a
relative Hopf $2$-cocycle problem for the remaining classes. Finally, we compute
multi-extended Jordanian examples, study a McConnell-type family of
quantum differential algebras, and establish a conditional birational
deletion criterion and a normal form when the unipotent radical has
dimension at most one.
\end{abstract}

\maketitle

\enlargethispage{2\baselineskip}
\tableofcontents

\section{Introduction}\label{sec:introduction}

Let $G$ be an {\bf affine algebraic group} over $\C$, and let $J$ be a Hopf
$2$-cocycle for its coordinate Hopf algebra $\Os(G)$.  Let ${}_J\mathscr O(G)_J$ be the
two-sided twisted Hopf algebra, and let $\Os(G)_J$ be the
one-sided twisted function algebra.  The latter algebra was
introduced and studied systematically in \cite{G}.  In particular, the
support of $J$ and the center of $\Os(G)_J$ were determined
\cite[Theorems~3.1 and~3.5]{G}, and for {\bf connected nilpotent} $G$,
$\Os(G)_J$ was shown to be a Noetherian domain and the simple algebras of the form $\Os(G)_J$
were described. In particular, {\bf minimal} Hopf
$2$-cocycles  for {\bf unipotent} affine algebraic groups $G$ give
Weyl algebras $\mathcal A_n(\mathbb C)$ \cite[Theorem~4.7]{G}, while the connected nilpotent case
is expressed by crossed products of quantum-torus, Weyl, and polynomial
parts \cite[Theorem~4.12]{G}.  The paper \cite{G} concluded with several
questions about {\bf arbitrary} affine algebraic groups $G$, including whether
minimality of $J$ is equivalent to simplicity of $\Os(G)_J$, which simple algebras can
occur as $\Os(G)_J$, whether $\operatorname{Prim}(\mathscr O(G)_J)
 =\operatorname{Max}(\mathscr O(G)_J)$  
and this set is a single $G$-orbit, and whether $\Os(G)_J$ is semiprime and semiprimitive \cite[Question 6.1]{G}.

Part II of this series \cite{GII} studies the algebra structure and
representation theory of the two-sided cotriangular Hopf algebras
${}_J\mathscr O(G)_J$ for connected nilpotent affine algebraic
groups. These algebras are affine Noetherian domains of
Gelfand--Kirillov dimension $\dim G$
\cite[Corollary~3.2 and Theorem~7.2]{GII}. In the unipotent case,
their representations are studied through twisted function
algebras on double cosets of the support
\cite[Sections~4--5]{GII} and \cite{BG}.
The distinction between the two-sided and one-sided constructions
is already substantial when $G$ is unipotent and $J$ is minimal:
writing $\mathfrak g:=\operatorname{Lie}(G)$ and $\dim G=2n$, one has
${}_J\mathscr O(G)_J\cong U(\mathfrak g)$ and
$\mathscr O(G)_J\cong \mathcal A_n(\mathbb C)$
as algebras, by \cite[Theorem~3.4]{GII} and
\cite[Theorem~4.7]{G}, respectively.
The present paper returns to the one-sided twisted function
algebras $\mathscr O(G)_J$ and the questions concerning them
raised in Part I.

The purpose of this paper is twofold. We first answer the simplicity
part of \cite[Question~6.1(1)]{G} and \cite[Questions~6.1(4)--(5)]{G} for arbitrary
affine algebraic groups over $\C$. We then extend the structural theory
of one-sided twists from connected nilpotent to connected solvable
groups. In the solvable case we prove Noetherianity and the domain
property, identify the torus-trivial Hopf $2$-cocycle classes, and study explicit
normal forms. The classification of all resulting simple algebras and
the general noncentral-torus Hopf $2$-cocycle problem remain open.

For a {\bf connected solvable} affine algebraic group
$G$, we use the notation
\[
U:=R_u(G),\qquad T:=G/U,\qquad
\mathfrak g:=\Lie(G),\qquad
\mathfrak u:=\Lie(U),\qquad
\Lambda:=X(T).
\]
The passage from nilpotent to solvable groups introduces two new
features.  First, the maximal torus need not be central, so a toral
Hopf
$2$-cocycle cannot in general be separated from the unipotent part by
convolution with an invariant Hopf
$2$-cocycle.  Second, the successive
one-dimensional factors of $U$ are normal in $G$, but need not be central.  The
classification of unipotent fiber functors in
\cite[Theorem~3.2]{EG1}, the computation of invariant Hopf
$2$-cocycles in \cite[Theorem~7.8]{EG2}, and the structure theorem for
homogeneous varieties of split solvable affine algebraic groups in
\cite[Theorem~1]{B} are the principal external inputs.

We first establish a few general results, which do not require solvability.
In Proposition~\ref{prop:growth-general} we show that one-sided Hopf
$2$-cocycle twisting preserves the word-growth filtration determined by a
finite-dimensional generating subcoalgebra. In particular, it
preserves Gelfand-Kirillov dimension.  In Theorem~\ref{thm:minimal-simple}
we prove that for every affine algebraic group $G$, a Hopf $2$-cocycle $J$ for $G$  
is minimal if and only if $\Os(G)_J$ is simple, which answers the simplicity part of \cite[Question 6.1(1)]{G}.  The proof uses the
braided-commutativity and Galois property of $\Os(G)_J$, together with
nondegeneracy of the $R$-form. In Theorem~\ref{thm:q615-main}
we prove that $\operatorname{Prim}(\mathscr O(G)_J)=\operatorname{Max}(\mathscr O(G)_J)$ 
and this set is a single $G$-orbit, which answers \cite[Question~6.1(5)]{G}. We then conclude in 
Corollary \ref{cor:q615-radical} that $\Os(G)_J$ is semiprime and semiprimitive; that is, 
$\Rad(\Os(G)_J)=0= \Jac(\Os(G)_J)$, which answers \cite[Question~6.1(4)]{G}.

We then assume that $G$ is a {\bf connected solvable} affine algebraic group.  Then every fiber functor
on $\Rep(G)$ is {\bf classical} by Proposition~\ref{prop:classical-fiber}.
If $J$ is a Hopf $2$-cocycle for $G$, with $H:=\Supp(J)$, then in Corollary~\ref{cor:center-solvable}, using
\cite[Theorem~3.5]{G} and \cite[Theorem~1]{B}, we give an algebra isomorphism
\[
Z(\Os(G)_J)\cong\Os(G/H)
\cong
\C[x_1,\ldots,x_k,y_1^{\pm1},\ldots,y_{\ell}^{\pm1}].
\]
In the main structural result, Theorem~\ref{thm:solvable-ore-structure},
we construct an iterated differential Ore presentation
\begin{equation}\label{eq:intro-ore-form}
\Os(G)_J
\cong
\C^\eta[\Lambda]
[z_1;\delta_1]\cdots[z_{n};\delta_{n}],
\end{equation}
where $n:=\dim U$, 
and then conclude in Corollary~\ref{cor:ring-properties} that
$\Os(G)_J$ is an affine left and right Noetherian domain of
Gelfand-Kirillov dimension $\dim G$.

We next study the Hopf $2$-cocycles themselves.  Restriction to the canonical
semisimple subcategory $\Rep(T)=\Rep(G/U)\subseteq \Rep(G)$ yields the toral invariant
\eqref{eq:res-ss}.  In Theorem~\ref{thm:solvable-torus-trivial}, obtained
from the classification of unipotent fiber functors
\cite[Theorem~3.2]{EG1}, we identify the kernel of this restriction with
\[
\bigl\{r\in\mathfrak g\wedge\mathfrak u\mid 
\CYB(r)=0\bigr\}/G.
\]
In Proposition~\ref{prop:support-hull} we show that the support of the
quantization $J_f(r)$ is the algebraic hull of the carrier of $r$.
Thus, minimality in this case means algebraic density of the carrier,
and not necessarily equality with $\mathfrak g$; the odd-dimensional example
following Corollary~\ref{cor:quasi-frobenius-hull} illustrates this
phenomenon.  Invariant Hopf $2$-cocycles are described in
\cite[Theorem~7.8]{EG2}.  In Theorems
\ref{thm:centrally-supported-fibers} and~\ref{thm:central-split}
we classify the class in which the toral restriction is supported on the
central diagonalizable subgroup.  In Proposition
\ref{prop:relative-toral-reduction} we reduce an arbitrary toral class to
a torus-trivial cocycle problem for a torally twisted Hopf algebra.

The Ore description \eqref{eq:intro-ore-form} motivates the search for
normal forms for the simple algebras arising from minimal Hopf $2$-cocycles for $G$.
In Example~\ref{ex:basic-solvable-blocks} we include the basic Jordanian
identity
\[
\Os(\Gm\ltimes\Ga)_J\cong\D(\Gm),
\]
while in Theorem~\ref{thm:extended-block} we compute a multi-extended
Jordanian family and give
\[
\Os(G)_J\cong\D(\Gm)\otimes \mathcal A_m(\C).
\]
These examples are based on the Jordanian, extended-Jordanian, chain,
and peripheric constructions of
\cite{KLM,KLO,LD1,LD2}; gauge-equivalent interpolating twists are
studied in \cite{BMP,MMSS}.  Proposition~\ref{prop:external-products}
produces larger examples by external products. In \S\ref{sec:normal-family} we study the family of algebras
$A(\Lambda,\eta;V,\ell,\omega)$, which simultaneously contain
quantum tori, Weyl algebras, rings of differential operators on tori,
and mixed Jordanian examples; its definition and auxiliary properties are given in
\S\ref{sec:prelim-quantum-differential}. In Theorem~\ref{thm:abelian-realization}
we realize every member of this family by an explicit Hopf $2$-cocycle on a
connected commutative affine algebraic group, and characterize simplicity by joint
nondegeneracy.  In Corollary~\ref{cor:solvable-dimension-two} we give a
complete classification in dimension two.

In \S\ref{sec:normal-family} we also develop the birational deletion
problem.  In Proposition~\ref{prop:normal-localization-collapse} we explain
why localization at normal elements of the final simple algebra is
trivial, while in Theorem~\ref{thm:deletion-criterion} we give an inductive
quotient-division-ring criterion.  
In Theorem \ref{thm:center-valued-outer} we describe derivations of the algebras
$A(\Lambda,\eta;V,\ell,\omega)$ modulo inner derivations.
Theorem~\ref{thm:rank-one} establishes the strong normal form for minimal $J$ when
$\dim R_u(G)=1$, and Corollary \ref{cor:rank-one-conjectures} includes the torus case.

The paper is organized as follows.
Section~\ref{sec:preliminaries} fixes the twisting conventions and collects
the auxiliary results on Galois objects, semi-invariants and affine charts,
solvable groups, quantum tori, normal elements and division rings, and
polynomial cohomology.
Section~\ref{sec:general-results} proves the general results on growth,
simplicity, primitive ideals, radicals, and finite-dimensional
irreducible representations.
Section~\ref{sec:solvable-cocycles} develops the solvable Ore structure,
the torus-trivial and centrally supported sectors, and the Jordanian
examples.
Section~\ref{sec:normal-family} studies the candidate family of algebras
$A(\Lambda,\eta;V,\ell,\omega)$, its abelian realization and derivations,
and the conditional deletion problem.

\medskip
\noindent
{\bf Acknowledgments.} This work was supported by Simons Foundation
Award 963288.

\section{Preliminaries}\label{sec:preliminaries}

Throughout the paper, we work over the field $\C$.  Unless otherwise stated, $G$ denotes an affine
algebraic group, and $\Os(G)$ its coordinate Hopf algebra. All algebras are nonzero and unital and all modules are unital.

\subsection{Hopf \texorpdfstring{$2$}{2}-cocycles and twisted function algebras}
\label{sec:prelim-cocycles}

Let $A$ be a Hopf algebra.  A {\bf Hopf $2$-cocycle} for $A$ is a
convolution-invertible linear map
$J:A\otimes A\longrightarrow\C$, such that
\begin{equation}\label{eq:hopf-2-cocycle}
J(a_{(1)}b_{(1)},c)J(a_{(2)},b_{(2)})
=
J(a,b_{(1)}c_{(1)})J(b_{(2)},c_{(2)}),\qquad J(a,1)=\eps(a)=J(1,a)
\end{equation}
for all $a,b,c\in A$.

The {\bf one-sided twisted algebra} $A_J$ is the vector space $A$ with
multiplication
\begin{equation}\label{eq:one-sided-cocycle-product}
a\cdot_J b
:=a_{(1)}b_{(1)}J(a_{(2)},b_{(2)}),\qquad a,b\in A.
\end{equation}
The {\bf two-sided twisted Hopf algebra} $A^J:={}_JA_J$
has the same coalgebra as $A$ and multiplication
\begin{equation}\label{eq:two-sided-cocycle-product}
a\star_J b
:=J^{-1}(a_{(1)},b_{(1)})a_{(2)}b_{(2)}
  J(a_{(3)},b_{(3)}),\qquad a,b\in A.
\end{equation}
When $A=\Os(G)$ is commutative, $A^J$ is cotriangular with
{\bf $R$-form}
\begin{equation}\label{eq:twisted-cotriangular-form}
R^J:=J_{21}^{-1}*J.
\end{equation}
We call $J$ {\bf minimal} if $R^J$ is nondegenerate, equivalently if
$A\rightarrow A^*$, $a\mapsto R^J(a,-)$, is injective.

Recall that {\bf gauge-equivalent} Hopf $2$-cocycles yield isomorphic one-sided twisted
algebras and isomorphic two-sided twisted Hopf algebras.  By
\cite[Theorem~3.1]{G}, every Hopf $2$-cocycle for $\Os(G)$ is gauge
equivalent to (the inflation of) a minimal Hopf $2$-cocycle on a closed subgroup
$H\subseteq G$, unique up to conjugacy.  This subgroup is the
{\bf support} of $J$, and is denoted by $\Supp(J)$.  In particular,
\begin{equation}\label{eq:minimal-support}
J\text{ is minimal if and only if }
\Supp(J)=G.
\end{equation}

The external-product construction will also be used in the examples.

\begin{proposition}\label{prop:external-products}
Let $(G_i,J_i)$ be affine algebraic groups with Hopf $2$-cocycles, $i=1,2$.  For
the external product Hopf $2$-cocycle $J_1\boxtimes J_2$ for $G_1\times G_2$, there is an algebra isomorphism
\[
   \Os(G_1\times G_2)_{J_1\boxtimes J_2}
   \cong
   \Os(G_1)_{J_1}\otimes\Os(G_2)_{J_2}.
\]
If $J_1,J_2$ are minimal, then $J_1\boxtimes J_2$ is minimal.
\end{proposition}

\begin{proof}
This follows since both the product formula, and the cotriangular form, factor tensorwise,
so nondegeneracy is preserved by external products.
\end{proof}

\subsection{Jordanian and extended-Jordanian twists}
\label{sec:jordanian-twists}

Assume that $G$ is {\bf connected and solvable}, and let
$\mathfrak g:=\Lie(G)$.  Recall that the right regular action of
$U(\mathfrak g)$ on $\Os(G)$ is given by
\begin{equation}\label{eq:right-regular-action}
(X\triangleright f)(a)
:=
\left.\frac{d}{ds}\right|_{s=0}f(a\exp(sX)),
\qquad
X\in\mathfrak g,
\quad f\in\Os(G),
\quad a\in G.
\end{equation}
Let
\[
\mathcal F=\sum_k u_k\otimes v_k
\]
be a {\bf Drinfeld twist} in a suitable completion of
$U(\mathfrak g)^{\otimes2}$. Thus, using the right-oriented twist convention, $\mathcal F$ is an invertible
normalized element in a suitable completion of $U(\mathfrak g)^{\otimes2}$ satisfying
\[
 (\Delta\otimes\id)(\mathcal F)(\mathcal F\otimes1)
 =(\id\otimes\Delta)(\mathcal F)(1\otimes\mathcal F),
 \qquad
 (\eps\otimes\id)(\mathcal F)
 =(\id\otimes\eps)(\mathcal F)=1.
\]
We require that $\mathcal F$ and $\mathcal F^{-1}$ define mutually inverse
operators on the algebraic tensor product $\Os(G)\otimes\Os(G)$. In the
Jordanian examples the relevant series terminate on each input in the
unipotent directions. Purely toral exponential factors are instead defined
by their scalar action on torus-weight spaces. Define 
\begin{equation}\label{eq:J-from-F}
J_{\mathcal F}(f,g)
:=
(\eps\otimes\eps)
\bigl(\mathcal F\triangleright(f\otimes g)\bigr),
\qquad f,g\in\Os(G).
\end{equation}
Then $J_{\mathcal F}$ is a Hopf $2$-cocycle for $\Os(G)$, and we have
\begin{equation}\label{eq:star-from-F}
f\cdot_{J_{\mathcal F}}g
=
m\bigl(\mathcal F\triangleright(f\otimes g)\bigr),
\qquad f,g\in\Os(G).
\end{equation}

Recall that the basic factors used in {\bf Jordanian-chain} constructions are
\begin{align}
\mathcal F_R^{\mathrm{std}}
&:=
\exp\!\left(\sum_{i,j}c_{ij}H_i\otimes H_j\right),
\label{eq:Reshetikhin-factor}\\
\mathcal F_J^{\mathrm{std}}
&:=
\exp(H\otimes\sigma),
\qquad
\sigma:=\log(1+\xi E),
\label{eq:Jordanian-factor}\\
\mathcal F_E^{\mathrm{std}}
&:=
\exp\bigl(\xi A\otimes Be^{-\beta\sigma}\bigr),
\label{eq:extension-factor}
\end{align}
where $\xi\in\C^\times$, $\alpha,\beta\in \C$, $A,B,E,H\in \mathfrak g$, and 
\[
[H,E]=E,
\qquad
[H,A]=\alpha A,
\qquad
[H,B]=\beta B,
\qquad
[A,B]=E,
\qquad
\alpha+\beta=1.
\]
Here the $H_i$ in \eqref{eq:Reshetikhin-factor} commute, and the
extended carrier also satisfies $[A,E]=[B,E]=0$.
The factor $\mathcal F_R^{\mathrm{std}}$ in
\eqref{eq:Reshetikhin-factor} is an {\bf abelian} (or
\emph{Reshetikhin}) twist, the factor $\mathcal F_J^{\mathrm{std}}$
in \eqref{eq:Jordanian-factor} is the {\bf basic Jordanian} twist, and 
the element $\mathcal F_E^{\mathrm{std}}$ in
\eqref{eq:extension-factor} is the {\bf extension factor}. The element 
\[
\mathcal F_{\mathrm{EJ}}^{\mathrm{std}}
:=
\mathcal F_E^{\mathrm{std}}\mathcal F_J^{\mathrm{std}}
=
\exp\bigl(\xi A\otimes Be^{-\beta\sigma}\bigr)
\exp(H\otimes\sigma)
\]
is the {\bf complete extended Jordanian} twist. 
More precisely, $\mathcal F_J^{\mathrm{std}}$ first deforms $\Delta$ to
\[
\Delta_J(a)
=
\mathcal F_J^{\mathrm{std}}\Delta(a)
(\mathcal F_J^{\mathrm{std}})^{-1},
\qquad a\in U(\mathfrak g).
\]
The extension factor $\mathcal F_E^{\mathrm{std}}$ is then a twist
for $\Delta_J$, so their ordered product is a twist for the original
coproduct; see \cite{KLM,KLO,LD1}. 

The displayed factors are written in the standard convention of the
cited literature. A standard twist $\mathcal F^{\mathrm{std}}$
satisfies
\[
 (\mathcal F^{\mathrm{std}}\otimes 1)(\Delta\otimes\id)
       (\mathcal F^{\mathrm{std}})
 =(1\otimes \mathcal F^{\mathrm{std}})(\id\otimes\Delta)
       (\mathcal F^{\mathrm{std}}).
\]
Its inverse satisfies our right-oriented convention and is the
operator inserted into \eqref{eq:J-from-F}. In the Jordanian examples
below we use $\mathcal F(\xi):=\mathcal F^{\mathrm{std}}(-\xi)^{-1}$
so that the first-order commutators have the displayed positive signs.
In particular, for the complete extended Jordanian twist this gives
\[
\begin{aligned}
\mathcal F_{\mathrm{EJ}}(\xi)
&:=\bigl(\mathcal F_{\mathrm{EJ}}^{\mathrm{std}}(-\xi)\bigr)^{-1}\\
&=\exp(-H\otimes\sigma_{-\xi})\,
  \exp\bigl(\xi A\otimes Be^{-\beta\sigma_{-\xi}}\bigr),
\qquad \sigma_{-\xi}:=\log(1-\xi E).
\end{aligned}
\]
Thus inversion reverses the order of the Jordanian and extension
factors.
The Jordanian and extended-Jordanian factors are constructed in
\cite{KLM}, chains for the classical Lie algebras are developed in
\cite{KLO}, and peripheric factors and their Jordanian or Reshetikhin
enlargements are studied in \cite{LD1,LD2}.  Coboundary interpolations
between Jordanian twists are analyzed in \cite{BMP,MMSS}.  A
coboundary modification produces a gauge-equivalent Hopf $2$-cocycle,
and therefore does not change the isomorphism class of the associated
one-sided twisted algebra.

\subsection{Primitive ideals, normal elements, and division-ring centers}
\label{sec:prelim-ring-theory}

Let $A$ be an associative algebra. Recall that  
the annihilators of simple left $A$-modules are called {\bf primitive ideals}. For example, every {\bf two-sided maximal} ideal of $A$ is primitive. Recall also that an element $a\in A$ is {\bf regular} if it is not a zero divisor and {\bf normal} if $Aa=aA$.

The following two observations concern simple rings and their Ore localizations.

\begin{lemma}\label{lem:normal-localization-trivial}
Let \(R\) be a simple ring, and let \(S\) be a left and right Ore set generated
by nonzero normal elements of \(R\).  Then
$S\subseteq R^\times$ and $S^{-1}R=R$ canonically.
\end{lemma}

\begin{proof}
If \(0\neq s\in R\) is normal, then \(Rs=sR\) is a nonzero two-sided
ideal of \(R\).  Since \(R\) is simple, \(Rs=sR=R\).  Hence, there are
\(a,b\in R\) such that \(as=1\) and \(sb=1\), which implies that
\[
a=a(sb)=(as)b=b,
\]
so \(s\) is a unit.  Thus every generator of \(S\), and hence every
element of \(S\), is a unit.  Since localization at units does not change the
ring, the claim follows.
\end{proof}

\begin{lemma}\label{lem:fraction-center}
If $R$ is a simple Ore domain and $D:={\rm Frac}(R)$, then $Z(D)=Z(R)$.
\end{lemma}

\begin{proof}
For $z\in Z(D)$, let $I_z:=\{a\in R\mid za\in R\}$. A common denominator
shows that $I_z\neq0$, and centrality of $z$ makes $I_z$ a two-sided
ideal of $R$. Simplicity gives $I_z=R$, so $z=z1\in R$.
The reverse inclusion is immediate.
\end{proof}

\subsection{Braided-commutative Galois objects}
\label{sec:prelim-galois}

Let $(H,R)$ be a minimal cotriangular Hopf algebra, and let $A$ be a right $H$-comodule algebra with coaction
$\rho(a)=a_{(0)}\otimes a_{(1)}$, $a\in A$. 
Recall (see e.g. \cite{Bi}) that $A$ is a right {\bf $H$-Galois object} over $\C$ if the left $A$-linear map   
\[
\gamma:A\otimes A\longrightarrow A\otimes H,
\qquad
\gamma(a\otimes b)=ab_{(0)}\otimes b_{(1)}
\]
is bijective. For $h\in H$, let
$h^{[1]}\otimes h^{[2]}:=
\gamma^{-1}(1\otimes h)$, so that
\begin{equation}\label{eq:translation-map-identity}
h^{[1]}h^{[2]}_{(0)}\otimes h^{[2]}_{(1)}
=
1\otimes h,\qquad h\in H.
\end{equation}

Recall also that $A$ is {\bf $R$-commutative} if 
\begin{equation}\label{eq:R-commutative}
ab
=
b_{(0)}a_{(0)}R(a_{(1)},b_{(1)}),
\qquad a,b\in A.
\end{equation}

\begin{lemma}
\label{lem:braided-galois-simple}
Let $(H,R)$ be a minimal cotriangular Hopf algebra, and let $A$ be $R$-commutative 
and right $H$-Galois.  Then $A$ is a simple algebra.
\end{lemma}

\begin{proof}
Let $I$ be a two-sided ideal of $A$, and let $a\in I$.  Then $h^{[1]}ah^{[2]}\in I$ for every
$h\in H$.  Thus, by
\eqref{eq:translation-map-identity}-\eqref{eq:R-commutative}, we have
\begin{align*}
h^{[1]}ah^{[2]}
&=
h^{[1]}h^{[2]}_{(0)}a_{(0)}
R\left(a_{(1)},h^{[2]}_{(1)}\right)
=
a_{(0)}R(a_{(1)},h)
\in I\qquad (h\in H).
\end{align*}
Writing $\rho(a)=\sum_{i=1}^n a_i\otimes c_i$,
where $c_1,\ldots,c_n\in H$ are linearly independent, we have
\begin{equation}\label{eq:R-linear-combinations}
\sum_{i=1}^nR(c_i,h)a_i\in I,
\qquad h\in H.
\end{equation}
Since $R$ is nondegenerate, the functionals
$R(c_i,-)$ are linearly independent, so we can choose
$h_1,\ldots,h_n$ in $H$, such that the matrix
$\bigl(R(c_i,h_j)\bigr)_{i,j}$ is invertible.  Thus, by
\eqref{eq:R-linear-combinations}, $a_i\in I$ for every $i$.
Therefore, $\rho(I)\subseteq I\otimes H$;
that is, the ideal $I$ is a right $H$-subcomodule. Thus, 
$\gamma(A\otimes I)\subseteq I\otimes H$, and since $\gamma$ is left $A$-linear and surjective,
we have $I\otimes H=\gamma(I\otimes A)$.
Thus, by injectivity of $\gamma$, we get
\begin{equation}\label{eq:tensor-ideal-inclusion}
A\otimes I\subseteq I\otimes A.
\end{equation}

Now if $I\neq0$, choose $0\neq x\in I$ and a linear functional
$\lambda:A\to\C$ with $\lambda(x)=1$. Then for every $b\in A$,
$b\otimes x\in I\otimes A$ by \eqref{eq:tensor-ideal-inclusion},
so applying $\id\otimes\lambda$ yields $b\in I$.  Thus $I=A$, as desired.
\end{proof}

\subsection{Support semi-invariants and affine charts}
\label{sec:prelim-support-charts}

Let $J$ be a Hopf $2$-cocycle for $G$ with support $H$. By the support theorem \cite[Theorem 3.1]{G}, we can choose a representative of the gauge class of $J$ of the form
$J=J_H\circ(\pi\otimes\pi)$, where $\pi:\Os(G)\twoheadrightarrow\Os(H)$
and $J_H\in(\Os(H)\otimes\Os(H))^*$ is a minimal Hopf $2$-cocycle for $H$.

We use the right $H$-coaction on $\Os(G)$
\[
 \rho:=(\id\otimes\pi)\Delta,\qquad
 \rho(f)(a,h)=f(ah).
\]
Write $\rho(a)=a_{(0)}\otimes a_{(1)}$, where
$a_{(1)}\in\Os(H)$. For the chosen support representative the product is
\[
 a\cdot_J b=a_{(0)}b_{(0)}J_H(a_{(1)},b_{(1)}).
\]
The left $G$-action on $\Os(G)_J$ is
$(g\cdot a)(x)=a(g^{-1}x)$. It commutes with the right support coaction.

A right $H$-semi-invariant of weight $\chi\in X(H)$ is a function $f\in \Os(G)$ satisfying
$\rho(f)=f\otimes\chi$, equivalently $f(ah)=\chi(h)f(a)$.
When $\chi$ is not specified, we will say that $f$ is {\bf $H$-semi-invariant}.

\begin{lemma}\label{lem:q615-normal-localization}
Let $\chi\in X(H)$, and let $f\in \mathscr O(G)$ be right $H$-semi-invariant of weight $\chi$.
Then the following hold:
\begin{enumerate}
\item
\begin{equation}\label{eq:q615-normal-ideals}
\mathscr O(G)_J\cdot_J f=f\mathscr O(G)=f\cdot_J \mathscr O(G)_J,
\end{equation}
as vector subspaces of $\mathscr O(G)$.
In particular, $f$ is {\bf normal} in $\mathscr O(G)_J$.  

\item
If moreover $f$ is regular in
$\mathscr O(G)$, it is also regular in $\mathscr O(G)_J$, and there is an algebra isomorphism 
\begin{equation}\label{eq:q615-localization}
\mathscr O(G)_J[f^{-1}]\cong\bigl(\mathscr O(G)[f^{-1}]\bigr)_{J},
\end{equation}
where the expression on the right means twisting of the right $\mathscr O(H)$-comodule
algebra $\mathscr O(G)[f^{-1}]$.
\end{enumerate}
\end{lemma}

\begin{proof}
(1) Define linear maps $T_\chi,U_\chi:\mathscr O(G)\longrightarrow \mathscr O(G)$ by
\[
 T_\chi(a)=a_{(0)}J_H(\chi,a_{(1)}),\qquad
 U_\chi(a)=a_{(0)}J_H(a_{(1)},\chi).
\]
Since $\chi\in \mathscr O(H)$ is grouplike, both maps are bijective with inverses obtained by replacing $J_H$
with $J_H^{-1}$.  Then for every $g\in \mathscr O(G)$, we have
\[
 f\cdot_J g=fT_\chi(g),\qquad g\cdot_J f=fU_\chi(g),
\]
where the products on the right are ordinary products in $\mathscr O(G)$.
This proves \eqref{eq:q615-normal-ideals}.  

(2) By \eqref{eq:q615-normal-ideals}, if $f$ is regular in
$\mathscr O(G)$, it is also regular in $\mathscr O(G)_J$. Since a normal regular element has a two-sided Ore localization, we may form the noncommutative Ore localization
$(\Os(G)_J)[f^{-1}]$. Separately, the right $H$-coaction extends to the
commutative localization $\Os(G)[f^{-1}]$ by
$$\rho(f^{-1})=f^{-1}\otimes\chi^{-1}.$$
Now since 
$$f\cdot_J f^{-1}=J_H(\chi,\chi^{-1})1
 =J_H(\chi^{-1},\chi)1=f^{-1}\cdot_J f,$$
and $J_H(\chi,\chi^{-1})
 =J_H(\chi^{-1},\chi)\ne 0$, it follows that $f$ is invertible in $\left(\mathscr O(G)[f^{-1}]\right)_J$.  Also, we have $a\cdot_J f^{-n}
 =U_{\chi^{-n}}(a)f^{-n}$ for every $n$. Since $U_{\chi^{-n}}$ is bijective, every ordinary fraction lies in the
algebra generated by $\mathscr O(G)_J$ and the twisted inverse of $f$.  Thus, \eqref{eq:q615-localization} follows from the Ore universal property and regularity of $f$.
\end{proof}

\begin{lemma}\label{lem:q615-affine-charts}
Let $q:G\to X:=G/H$ be the quotient map.  There are finitely many
nonzero right $H$-semi-invariants $f_1,\ldots,f_s\in\mathscr O(G)$,
and affine open subsets $X_1,\ldots,X_s$ covering $X$, such that
\[
 q^{-1}(X_i)=G_{f_i}:=\{a\in G\mid f_i(a)\ne0\}.
\]
If moreover $G=G^\circ H$, then each $f_i$ is regular in $\mathscr O(G)$.
\end{lemma}

\begin{proof}
By Chevalley's theorem, there are a finite-dimensional rational
$G$-module $W$ and a line $\mathbb C v\subset W$ whose stabilizer is
$H$.  Thus, $X$ is the locally closed orbit $G\cdot [v]\subseteq\mathbb P(W)$
(see e.g. \cite[Theorems~4.27 and~7.18]{Mi}).
Let $\overline X$ be the projective closure of $X$, and let
$D:=\overline X\setminus X$.  For each $x\in X$, choose a 
homogeneous polynomial $p$ of positive degree vanishing on $D$, and satisfying  $p(x)\ne 0$.  Then
$X_p:=\overline X\cap D_+(p)\subseteq X$
is affine open, and contains $x$.

Let $\chi\in X(H)$ such that $hv=\chi(h)v$ for every $h\in H$.  If $p$ has degree $d$, the function
$f_p(a):=p(av)$, $a\in G$, is regular and satisfies for every $h\in H$,
\[
 f_p(ah)=\chi(h)^d f_p(a),\qquad q^{-1}(X_p)=G_{f_p}.
\]
Finally, since $X$ is quasi-compact, there are finitely many such charts covering $X$.

If moreover $G=G^\circ H$, then right multiplication by $H$ permutes the
connected components of $G$ transitively.  Therefore a nonzero right
$H$-semi-invariant cannot vanish identically on any component, and since each component is irreducible and reduced, the semi-invariant is regular in $\mathscr O(G)$.
\end{proof}

\begin{lemma}\label{lem:q615-fiber}
Let $X_f\subseteq X=G/H$ be one of the affine charts in
Lemma~\ref{lem:q615-affine-charts}, and assume that $f$ is regular. Then the following hold:
\begin{enumerate}
\item
$\mathscr O(X_f)$ embeds in the center of the algebra $(\mathscr O(G)_J)[f^{-1}]$.  
\item
For each $x:=aH\in X_f(\mathbb C)$,
with maximal ideal $\mathfrak m_x\subset \mathscr O(X_f)$, there is an algebra isomorphism
\begin{equation}
\label{eq:q615-simple-fiber}
(\mathscr O(G)_J)[f^{-1}]/\mathfrak m_x(\mathscr O(G)_J)[f^{-1}]\cong\mathscr O(H)_{J_H}.
\end{equation}
\end{enumerate}
\end{lemma}

\begin{proof}
Pullback along $q$ identifies $\Os(X_f)$ with the subalgebra $\mathscr O(G_f)^H\subseteq \mathscr O(G_f)$. For
$c\in\Os(X_f)$ and $a\in\Os(G_f)$, we have $\rho(c)=c\otimes1$ and
$c\cdot_Ja=ca=a\cdot_Jc$.
Thus $\Os(X_f)$ embeds centrally in $(\Os(G)_J)[f^{-1}]$, and
multiplication by this subalgebra is unchanged by twisting. In particular,
for $x:=aH$, we have
\[
 (\Os(G)_J)[f^{-1}]/\mathfrak m_x(\Os(G)_J)[f^{-1}]
 \cong
 \bigl(\Os(G_f)/\mathfrak m_x\Os(G_f)\bigr)_{J_H}.
\]
The map $G_f\to X_f$ is the restriction of the right $H$-torsor
$G\to G/H$, and its scheme-theoretic fiber at $x=aH$ is $aH$.
Thus, the map $H\to aH$, $h\mapsto ah$, identifies its coordinate algebra with $\mathscr O(H)$,
as a right $H$-comodule algebra, and twisting this
identification proves \eqref{eq:q615-simple-fiber}.
\end{proof}

\subsection{Normal series and characters of solvable groups}
\label{sec:prelim-solvable-groups}

In this subsection, let $G$ be {\bf connected solvable}, let $U:=R_u(G)$
and $T:=G/U$, and let $\mathfrak g:=\Lie(G)$ and
$\mathfrak u:=\Lie(U)$.

\begin{lemma}\label{lem:normal-series}
Let $n:=\dim U$. There is a sequence of closed connected subgroups
\[
   U=U_0\supset U_1\supset\cdots\supset U_n=1,
\]
such that $U_i$ is normal in $G$, and
$U_{i-1}/U_i\cong\Ga$, $1\leq i\leq n$.
\end{lemma}

\begin{proof}
The adjoint action of the solvable Lie algebra $\mathfrak g$ on its ideal
$\mathfrak u$ admits a complete invariant flag by Lie's theorem.  Namely,
there is a flag of $\mathfrak g$-stable ideals
\[
   0=\mathfrak u_0\subset\mathfrak u_1\subset\cdots
   \subset\mathfrak u_n=\mathfrak u,
   \qquad \dim(\mathfrak u_i/\mathfrak u_{i-1})=1\qquad (1\leq i\leq n).
\]
Since for a unipotent group over $\mathbb{C}$, Lie subalgebras correspond
to closed connected subgroups, and $\mathfrak g$-stable ideals correspond
to normal subgroups in $G$, reversing this flag gives the required
series.
\end{proof}

Choose $U_0,\dots,U_n$ as in Lemma \ref{lem:normal-series}, and let
$G_i:=G/U_i$, $0\leq i\leq n$.
Then $G_0=T$, $G_n=G$, and for every $1\leq i\leq n$ there is an exact sequence
\[
   1\longrightarrow \Ga
   \longrightarrow G_i
   \longrightarrow G_{i-1}
   \longrightarrow1.
\]
Set $A_i:=\Os(G_i)$.  Then
\[
   A_0\subset A_1\subset\cdots\subset A_n=\Os(G)
\]
is a chain of Hopf subalgebras.

\begin{lemma}
\label{lem:triangular-coordinate}
For every $1\le i\le n$, there exists $z_i\in A_i$ such that
$ A_i=A_{i-1}[z_i]$
as algebras, and
\begin{equation}\label{eq:triangular-coproduct}
   \Delta(z_i)
   =z_i\otimes1+\chi_i\otimes z_i+w_i,
\end{equation}
where $\chi_i\in A_{i-1}$ is grouplike and
$w_i\in A_{i-1}\otimes A_{i-1}$.
\end{lemma}

\begin{proof}
The quotient morphism $G_i\twoheadrightarrow G_{i-1}$ is a $\Ga$-torsor.  Since
$G_{i-1}$ is affine, $H^1(G_{i-1},\Ga)=0$,
so this torsor is trivial.  Thus, we can choose a (normalized at the
identity) regular section $G_{i-1}\xrightarrow{1:1} G_i$, and make the identification $G_i=\Ga\times G_{i-1}$ as varieties.

Since $\Aut(\Ga)=\Gm$, conjugation by $G_{i-1}$ on the kernel $\Ga$ is given by a
character
$\chi_i:G_{i-1}\rightarrow\Gm$.
Thus, in the chosen coordinates, the multiplication of $G_i=\Ga\times G_{i-1}$
has the form
\[
   (x,g)(y,h)
   =\bigl(x+\chi_i(g)y+c_i(g,h),gh\bigr),\qquad x,y\in \Ga,\quad g,h\in G_{i-1},
\]
for a regular normalized function
$c_i:G_{i-1}\times G_{i-1}\to\Ga$.  Then taking $z_i(x,g)=x$ gives
\eqref{eq:triangular-coproduct}, with $w_i$ the function corresponding to
$c_i$.
\end{proof}

\begin{lemma}\label{lem:characters}
Let $G$ be connected solvable and $S\subseteq G$ a closed
subgroup. Then $S_u:=S\cap U$ is the unipotent radical of $S$, the quotient
$D:=S/S_u$ is a diagonalizable subgroup of $T=G/U$, and restriction
$X(G)\to X(S)$ is surjective.
\end{lemma}

\begin{proof}
Every closed subgroup of a unipotent complex algebraic group is connected
and unipotent. The quotient $S/(S\cap U)$ identifies with the closed image
of $S$ in the torus $T=G/U$, so it is diagonalizable. Thus $S\cap U$ is exactly
the unipotent radical of $S$.

It remains to prove that every algebraic character of $S$ extends to
$G$. Let
\[
q:G\longrightarrow T:=G/U,
\qquad
p:S\longrightarrow D:=S/(S\cap U)
\]
be the quotient maps, and let $j:D\hookrightarrow T$ be the closed
embedding induced by $q|_S$. Thus $q|_S=j\circ p$.

Let $\chi\in X(S)$. Since a unipotent algebraic group has no
nontrivial algebraic characters, $\chi$ is trivial on $S\cap U$.
Thus, it factors uniquely through a character
$\overline\chi\in X(D)$; that is, $\chi=\overline\chi\circ p$.

By the exact contravariant equivalence
between diagonalizable algebraic groups and finitely generated abelian
groups \cite[Theorem~12.9]{Mi}, the exact sequence
\[
1\longrightarrow D\xrightarrow{\,j\,}T\longrightarrow T/D
\longrightarrow1
\]
induces the exact sequence of character groups
\[
0\longrightarrow X(T/D)\longrightarrow X(T)
\xrightarrow{\,j^*\,}X(D)\longrightarrow0.
\]
Thus $j^*X(T)=X(D)$.

Choose $\widetilde\chi\in X(T)$ such that
$j^*(\widetilde\chi)=\overline\chi$, and define
$\widehat\chi:=\widetilde\chi\circ q\in X(G)$.
Then
\[
\widehat\chi|_S
=
\widetilde\chi\circ j\circ p
=
\overline\chi\circ p
=
\chi.
\]
Therefore the restriction map $X(G)\to X(S)$ is surjective.
Equivalently, restriction is the composite
$X(G)\cong X(T)\twoheadrightarrow X(D)\cong X(S)$.
\end{proof}

\subsection{Quantum tori and homogeneous derivations}
\label{sec:prelim-quantum-tori}

Let $\Lambda\cong\mathbb Z^r$ be a lattice.
For a group \(2\)-cocycle
\(\eta\in Z^2(\Lambda,\mathbb C^\times)\), let
\begin{equation}\label{eq:algebraic-quantum-torus}
\mathbb C^\eta[\Lambda]
:=
\bigoplus_{\lambda\in\Lambda}\mathbb C X^\lambda,
\qquad
X^\lambda X^\mu
=
\eta(\lambda,\mu)X^{\lambda+\mu}
\end{equation}
be the associated {\bf algebraic quantum torus}. Write
\[
q_\eta(\lambda,\mu):=\eta(\lambda,\mu)\eta(\mu,\lambda)^{-1}
\]
for its commutator bicharacter. Let
\[
   {\rm rad}(q_\eta)
   :=
   \{d\in\Lambda\mid q_\eta(d,\lambda)=1,\,
      \forall\lambda\in\Lambda\}.
\]

The following homogeneous-derivation calculation is the ordinary-derivation
case of the quantum-torus analysis in \cite{J}. We include the proof to
fix the cocycle and grading conventions used here.

\begin{lemma}\label{lem:homogeneous-derivations}
Let $\C^\eta[\Lambda]$ 
be a quantum torus and let $d\in\Lambda$.  Suppose that $\delta$ is a
homogeneous derivation of $\C^\eta[\Lambda]$ of degree $d$, namely
$\delta(X^\lambda)\in\C X^{\lambda+d}$ for every
$\lambda\in\Lambda$.
Then the following hold:
\begin{enumerate}
\item If $d\notin{\rm rad}(q_\eta)$, then $\delta$ is inner, induced by a
homogeneous element of degree $d$.
\item If $d\in{\rm rad}(q_\eta)$, then there are a central homogeneous unit
$u_d\in\C^\eta[\Lambda]_d$ and
$\alpha\in\Hom_{\mathbb Z}(\Lambda,\C)$, such that
$\delta=u_d\partial_\alpha$ and $\partial_\alpha(X^\lambda)=\alpha(\lambda)X^\lambda$. 
Once $u_d$ is fixed, $\alpha$ is unique.
\end{enumerate}
\end{lemma}

\begin{proof}
Since $\Lambda$ is free abelian, the $2$-cocycle $\eta$ is cohomologous to
a bicharacter.  Equivalently, after replacing each homogeneous basis
element $X^\lambda$ by a nonzero scalar multiple $Y^\lambda$, we may
assume that
\[
   Y^\lambda Y^\mu
   =\eta'(\lambda,\mu)Y^{\lambda+\mu}\qquad (\lambda,\mu\in \Lambda)
\]
for a bicharacter $\eta'$.  This graded algebra isomorphism preserves
the property that $\delta$ is homogeneous of degree $d$ and it does not
change the commutator bicharacter $q_\eta$.  Thus we may work in this basis, and
for notational simplicity, again write $X^\lambda$ and $\eta$.
Write
\[
   \delta(X^\lambda)=c(\lambda)X^{\lambda+d},\qquad c(\lambda)\in \mathbb{C}\qquad (\lambda\in \Lambda).
\]
Applying the derivation identity to
$X^\lambda X^\mu=\eta(\lambda,\mu)X^{\lambda+\mu}$ gives
\begin{equation}\label{eq:cocycle-equation-c}
   c(\lambda+\mu)
   =c(\lambda)\eta(d,\mu)
    +c(\mu)\eta(\lambda,d)\qquad (\lambda,\mu\in \Lambda).
\end{equation}
Setting
\[
   f(\lambda):=
   \frac{c(\lambda)}{\eta(\lambda,d)},
   \qquad
   \rho(\lambda):=q_\eta(d,\lambda)\qquad (\lambda\in \Lambda),
\]
we see that \eqref{eq:cocycle-equation-c} becomes
\begin{equation}\label{eq:rho-cocycle}
   f(\lambda+\mu)=f(\lambda)\rho(\mu)+f(\mu)\qquad (\lambda,\mu\in \Lambda).
\end{equation}
Since $\Lambda$ is abelian, interchanging $\lambda$ and $\mu$ yields
\begin{equation}\label{eq:rho-commute}
   f(\lambda)(\rho(\mu)-1)
   =f(\mu)(\rho(\lambda)-1)\qquad (\lambda,\mu\in \Lambda).
\end{equation}

Now if $d\notin{\rm rad}(q_\eta)$, choose $\nu$ with $\rho(\nu)\neq1$.  Then
by \eqref{eq:rho-commute}, we have
\[
   f(\lambda)=a(\rho(\lambda)-1)\qquad (\lambda\in \Lambda),
   \qquad
   a:=\frac{f(\nu)}{\rho(\nu)-1}.
\]
Thus for every $\lambda\in \Lambda$, we have
\begin{align*}
   \delta(X^\lambda)
   &=a\eta(\lambda,d)
     (q_\eta(d,\lambda)-1)X^{\lambda+d}=[aX^d,X^\lambda],
\end{align*}
so $\delta=\ad(aX^d)$.  Transporting back through the graded
rescaling gives an inner derivation induced by a homogeneous element of
degree $d$.

If $d\in{\rm rad}(q_\eta)$, then $\rho=1$, so
\eqref{eq:rho-cocycle} says that $f$ is additive.  Let $\alpha:=f$.
Since $X^d$ is central and
$\eta(d,\lambda)=\eta(\lambda,d)$, we obtain
\[
   X^d\partial_\alpha(X^\lambda)
   =\alpha(\lambda)\eta(d,\lambda)X^{d+\lambda}
   =c(\lambda)X^{\lambda+d}\qquad (\lambda\in \Lambda).
\]
In the rescaled basis, $X^d$ is a central homogeneous unit.  Transporting
back gives the asserted central homogeneous unit $u_d$.  Finally, uniqueness of
$\alpha$ after fixing $u_d$ is clear.
\end{proof}

\subsection{Jordanian quantum differential algebras and canonical derivations}
\label{sec:prelim-quantum-differential}

We introduce the family $\mathscr A$ of algebras used in
\S\ref{sec:normal-family}, together with its center and the canonical
derivations.

\begin{definition}\label{def:QDO}
Fix data $(\Lambda,\eta;V,\ell,\omega)$, where
\begin{itemize}
\item $\Lambda\cong\mathbb Z ^r$ is a lattice, $\eta\in Z^2(\Lambda,\C^\times)$ is a normalized $2$-cocycle,
\item $V$ is a finite-dimensional vector space, $\omega\in\wedge^2V^*$, and
\item
$\ell:V\longrightarrow\Hom_{\mathbb Z}(\Lambda,\C)$ 
is a linear map.
\end{itemize}
Define the Jordanian quantum differential algebra $A(\Lambda,\eta;V,\ell,\omega)$
to be the unital
$\C$-algebra generated by $X^\lambda$ ($\lambda\in\Lambda$) and
$P_v$ ($v\in V$, linearly in $v$), with $X^0=1$ and relations
\begin{equation}\label{eq:QDO-torus}
   X^\lambda X^\mu
   =\eta(\lambda,\mu)X^{\lambda+\mu},
   \end{equation}
\begin{equation}\label{eq:QDO-mixed}
   [P_v,X^\lambda]
   =\ell(v)(\lambda)X^\lambda,
\end{equation}
\begin{equation}\label{eq:QDO-additive}
   [P_v,P_w]
   =\omega(v,w).
\end{equation}
In particular,
$(X^\lambda)^{-1}=\eta(\lambda,-\lambda)^{-1}X^{-\lambda}$. \qed
\end{definition}

\begin{remark}
For $\eta=1$ and faithful lattice action, these algebras are instances of
McConnell's construction \cite[Section~4]{M}; arbitrary $\eta$ gives
its scalar-cocycle-twisted analogue. The quantum torus subcase is classical in \cite{MP}. The commutative-group case of \cite[Lemma~4.11 and Theorem~4.12]{G} and \cite[Theorem~3.4]{EG2} provides the full cocycle realization. Thus the phrase ``Jordanian quantum differential algebra'' is terminology we use, not a claim
to a newly discovered family. \qed 
\end{remark}

Note that after choosing a basis of $V$, $ A(\Lambda,\eta;V,\ell,\omega)$ is an iterated differential Ore
extension of $\C^\eta[\Lambda]$, and it has the expected PBW basis.

Let $q_\eta(\lambda,\mu)
   :=\eta(\lambda,\mu)\eta(\mu,\lambda)^{-1}$ 
be the alternating bicharacter on $\Lambda$ associated with $\eta$, and define the two {\bf joint
radicals}
\begin{equation}
   K_V
   :={}
   \{v\in V\mid \ell(v)=0,\ \omega(v,V)=0\},
   \label{eq:KV}
   \end{equation}
   \begin{equation}
   K_\Lambda
   :={}
   \{\lambda\in\Lambda\mid 
     q_\eta(\lambda,\mu)=1\ \forall\mu\in\Lambda,
     \ \ell(V)(\lambda)=0\}.
   \label{eq:KLambda}
\end{equation}
Note that if $K_V\neq0$ then some $P_v$ is central, and if
$K_\Lambda\neq0$ then some $X^\lambda$ is central. The following lemma
records the full center.

\begin{lemma}\label{lem:center}
Let $A:=A(\Lambda,\eta;V,\ell,\omega)$, $R:={\rm rad}~\omega$,
$R_0:=R\cap\ker\ell$, and
\[
 \Lambda_0:=\{\lambda\in \Lambda\mid q_\eta(\lambda,\mu)=1\ (\mu\in\Lambda),\
                       \ell(V)(\lambda)=0\}.
\]
Then
\[
 Z(A)=\bigoplus_{\lambda\in\Lambda_0}{\rm Sym}(R_0)X^\lambda.
\]
In particular, $Z(A)=\C$ if and only if $K_V=0=K_\Lambda$.
\end{lemma}

\begin{proof}
Let $B\subseteq A$ be the subalgebra generated by the $P_v$. Then the algebra $A=\bigoplus_{\lambda\in\Lambda}BX^\lambda$ is graded. Choosing a symplectic complement to
$R$ gives $B\cong \mathcal A_m(\C)\otimes{\rm Sym}(R)$ as algebras. 
For a central nonzero homogeneous element $bX^\lambda\in BX^\lambda$, commutation with
$P_v$ gives
\[
 [P_v,b]+\ell(v)(\lambda)b=0.
\]
Since $\ad(P_v)$ is locally nilpotent on $B$, its only possible eigenvalue
on a nonzero vector is zero. Thus $\ell(V)(\lambda)=0$ and $b\in Z(B)={\rm Sym}(R)$.
Conjugation by $X^\mu$ acts on $B$ by
$\tau_\mu(P_v)=P_v-\ell(v)(\mu)$. Centrality gives
$\tau_\mu(b)=q_\eta(\lambda,\mu)b$. Since translation is unipotent on each
finite-degree polynomial space, $q_\eta(\lambda,\mu)=1$ and $b$ is invariant
under all these translations. Their invariant polynomial ring is
${\rm Sym}(R_0)$. This proves the claimed formula for $Z(A)$.
\end{proof}

Call an algebra $A(\Lambda,\eta;V,\ell,\omega)$ 
{\bf jointly nondegenerate} if $K_V=0=K_\Lambda$.

Let $A:= A(\Lambda,\eta;V,\ell,\omega)$.
For $\alpha\in\Hom_{\mathbb Z}(\Lambda,\C)$ and $\beta\in V^*$,
define a linear map $\partial_{\alpha,\beta}:A\to A$ on generators by
\begin{equation}\label{eq:canonical-derivation}
   \partial_{\alpha,\beta}(X^\lambda)
   =\alpha(\lambda)X^\lambda,
   \qquad
   \partial_{\alpha,\beta}(P_v)
   =\beta(v).
\end{equation}

\begin{lemma}\label{lem:canonical-derivation}
The map \eqref{eq:canonical-derivation} extends uniquely to a derivation
of $A$ and hence to a derivation of $\operatorname{Frac} (A)$.
\end{lemma}

\begin{proof}
The torus relation \eqref{eq:QDO-torus} is preserved since $\alpha$ is
additive.  Applying $\partial_{\alpha,\beta}$ to
\eqref{eq:QDO-mixed} gives
\[
   \alpha(\lambda)
   \bigl([P_v,X^\lambda]-\ell(v)(\lambda)X^\lambda\bigr)=0
\]
and applying it to \eqref{eq:QDO-additive} gives $0=0$.  Thus, the map is
compatible with all defining relations.  Since a derivation of an Ore domain
extends uniquely to its quotient division ring by
$\partial(a^{-1})=-a^{-1}\partial(a)a^{-1}$,
we are done. 
\end{proof}

Retain the notation $R$, $R_0$, and $\Lambda_0$ from
Lemma~\ref{lem:center}.
For
\[
f\in\operatorname{Hom}_{\mathbb Z}(\Lambda,Z(A)),
\qquad
g\in\operatorname{Hom}_{\mathbb C}(V,Z(A)),
\]
define the center-valued canonical derivation
\[
\partial_{f,g}(X^\lambda):=f(\lambda)X^\lambda,
\qquad
\partial_{f,g}(P_v):=g(v).
\]

\subsection{Koszul complexes, polynomial differential forms, and translation cocycles}
\label{sec:prelim-polynomial-cohomology}

Let $\Lambda\cong\mathbb Z^r$ be a lattice.
Choose a basis $e_1,\ldots,e_r$ of $\Lambda$, let
$E:=\Lambda\otimes_{\mathbb Z}\mathbb C$, and let
$\varepsilon_1,\ldots,\varepsilon_r$ be the dual basis of $E^*$. Recall that for any commuting linear endomorphisms $T_1,\ldots,T_r$ of a vector space $M$, their {\bf cochain Koszul complex} is
\[
K^q(M;T):=M\otimes\wedge^q E^*,
\qquad 0\le q\le r,
\]
with differential
\begin{equation}\label{eq:L418-koszul-differential}
d_T(m\otimes\theta)
:=\sum_{i=1}^rT_i(m)\otimes(\varepsilon_i\wedge\theta).
\end{equation}
Its cohomology is
\[
H^q(K^\bullet(M;T))
:=\ker(d_T:K^q\to K^{q+1})/
  \operatorname{im}(d_T:K^{q-1}\to K^q).
\]
In the first two degrees the differential is
$d_T(m)=\sum_{i=1}^r T_i(m)\varepsilon_i$ and 
\[
d_T\left(\sum_{i=1}^r a_i\varepsilon_i\right)
=\sum_{1\le i<j\le r}\bigl(T_i(a_j)-T_j(a_i)\bigr)
                 \varepsilon_i\wedge\varepsilon_j.
\]
Thus a tuple $(a_1,\ldots,a_r)$ is a $1$-cocycle precisely when
$T_i(a_j)=T_j(a_i)$ for all $i,j$, and it is a $1$-coboundary
precisely when $a_i=T_i(h)$ for one $h\in M$.

Let $B$ be a commutative $\C$-algebra. Recall that the {\bf relative algebraic de Rham
complex} of a polynomial algebra $M=B[x_1,\ldots,x_s]$ over $B$ is
\[
\Omega^q_{M/B}
:=\bigoplus_{1\le i_1<\cdots<i_q\le s}
M\,dx_{i_1}\wedge\cdots\wedge dx_{i_q},
\qquad \Omega^0_{M/B}:=M,
\]
with differential 
\begin{equation}\label{eq:L418-derham-differential}
d_{\mathrm{dR}}
\bigl(a\,dx_{i_1}\wedge\cdots\wedge dx_{i_q}\bigr)
:=\sum_{j=1}^s\frac{\partial a}{\partial x_j}
\,dx_j\wedge dx_{i_1}\wedge\cdots\wedge dx_{i_q}.
\end{equation}
In particular, $d_{\mathrm{dR}}(b)=0$ for $b\in B$.

We recall a direct proof of the exactness of $\Omega^q_{M/B}$ in positive degrees.
Let
\[
\mathcal E:=\sum_{j=1}^s x_j\frac{\partial}{\partial x_j}.
\] 
Its contraction on differential forms is the $M$-linear map
\[
\begin{split}
\iota_{\mathcal E}
\bigl(a\,dx_{i_1}\wedge\cdots\wedge dx_{i_q}\bigr)
:=\sum_{j=1}^q(-1)^{j-1}x_{i_j}a\,
 dx_{i_1}\wedge\cdots\wedge\widehat{dx_{i_j}}
                   \wedge\cdots\wedge dx_{i_q}.
\end{split}
\]
If the coefficients of a $q$-form $\zeta$ are homogeneous of
$x$-degree $m$, then a calculation using
\eqref{eq:L418-derham-differential} gives
\begin{equation}\label{eq:L418-euler-homotopy}
(d_{\mathrm{dR}}\iota_{\mathcal E}
 +\iota_{\mathcal E}d_{\mathrm{dR}})(\zeta)
=(m+q)\zeta.
\end{equation}
Indeed, the operator on the left acts on a coefficient as $\mathcal E$, giving $m$, and acts on each factor $dx_j$ as the identity,
giving the additional $q$.

If $q>0$ and $d_{\mathrm{dR}}\zeta=0$, then $m+q$ is a nonzero
complex number and
\[
\zeta=d_{\mathrm{dR}}
\left(\frac{1}{m+q}\iota_{\mathcal E}\zeta\right).
\]
A general polynomial form is a finite sum of homogeneous forms;
closedness holds separately for their homogeneous components.
Applying this formula to each component proves exactness in every
positive degree. In degree zero, a polynomial with all its
$x$-derivatives zero belongs to $B$. Thus
\begin{equation}\label{eq:L418-derham-cohomology}
H^q(\Omega^\bullet_{M/B})
=\begin{cases}
B,&q=0,\\
0,&q>0.
\end{cases}
\end{equation}
This is the {\bf polynomial Poincare lemma} over $B$.

\begin{lemma}\label{lem:weyl}
For a finite-dimensional vector space $R$,
\[
 {\rm Der}_\C\bigl(\mathcal A_m(\C)\otimes{\rm Sym}(R)\bigr)
 =\operatorname{InnDer}\bigl(\mathcal A_m(\C)\otimes{\rm Sym}(R)\bigr)
   +{\rm Der}_\C({\rm Sym}(R)),
\]
where the second summand is extended to be zero on fixed Weyl generators.
\end{lemma}

\begin{proof}
A derivation preserves the center ${\rm Sym}(R)$. Subtract its restriction extended
by zero on the Weyl generators. The remaining derivation is ${\rm Sym}(R)$-linear.
For Weyl generators $p_i,q_i$ with $[p_i,q_j]=\delta_{ij}$, the operators
$[p_i,-]$ and $[-,q_i]$ are the PBW partial derivatives in $q_i$ and $p_i$.
Applying the derivation to the Weyl relations gives the closedness equations
for a polynomial one-form over ${\rm Sym}(R)$. The polynomial Poincare lemma,
valid over this characteristic-zero coefficient ring, supplies a potential.
Its commutator has the prescribed values on all Weyl generators and is the
required inner derivation.
\end{proof}

\begin{lemma}\label{lem:lattice}
Let $R$ be a finite-dimensional complex vector space, let
$\Lambda\cong\mathbb Z^r$, and let
\[
\ell:R\longrightarrow
\operatorname{Hom}_{\mathbb Z}(\Lambda,\mathbb C)
\]
be linear. Write $P_v$ for the degree-one element of
$\operatorname{Sym}(R)$ corresponding to $v\in R$. Let $\Lambda$
act on $M:=\operatorname{Sym}(R)$ by the translations
\[
\tau_\lambda(P_v)=P_v-\ell(v)(\lambda)1,
\]
and let
\[
R_0:=\{v\in R\mid \ell(v)(\lambda)=0,\,
                 \forall\lambda\in\Lambda\}.
\]
Suppose that $c:\Lambda\to M$ satisfies
\begin{equation}\label{eq:L418-cocycle}
c_{\lambda+\mu}=c_\lambda+\tau_\lambda(c_\mu)
\qquad(\lambda,\mu\in\Lambda).
\end{equation}
Then there exist $b\in M$, finitely many
$f_j\in\operatorname{Sym}(R_0)$, and
$\alpha_j\in\operatorname{Hom}_{\mathbb Z}(\Lambda,\mathbb C)$
such that
\begin{equation}\label{eq:L418-conclusion}
c_\lambda=b-\tau_\lambda(b)+\sum_jf_j\alpha_j(\lambda)
\qquad(\lambda\in\Lambda).
\end{equation}
\end{lemma}

\begin{proof}
Choose a basis $e_1,\ldots,e_r$ of $\Lambda$, let
$E:=\Lambda\otimes_{\mathbb Z}\mathbb C$, and let
$\varepsilon_1,\ldots,\varepsilon_r$ be the dual basis of $E^*$.
Consider the commuting difference operators
$\Delta_i:=\tau_{e_i}-\operatorname{id}_M$.
Set $c_i:=c_{e_i}$. Applying \eqref{eq:L418-cocycle} to
$e_i+e_j=e_j+e_i$ gives
$\Delta_i(c_j)=\Delta_j(c_i)$.
Thus,
$\mathbf c:=\sum_i c_i\varepsilon_i$ is a $1$-cocycle in $K^\bullet(M;\Delta)$.

Note that this complex computes the usual group cohomology $H^1(\Lambda,M)$. Indeed,
conversely, any tuple satisfying these compatibility equations
determines a group $1$-cocycle. To see this, form the semidirect
product of the additive group of $M$ with $\Lambda$, with multiplication
\[
(a,\lambda)(b,\mu)
:=(a+\tau_\lambda(b),\lambda+\mu).
\]
The compatibility equations say exactly that the elements
$(c_i,e_i)$ commute in $M\rtimes\Lambda$. Since $\Lambda$ is free abelian, they define a
unique homomorphism
\[
\Lambda\longrightarrow M\rtimes\Lambda,
\qquad \lambda\longmapsto(c_\lambda,\lambda),
\]
whose first coordinate satisfies \eqref{eq:L418-cocycle}.
Moreover, the tuple $(\Delta_i h)_i$ corresponds to the group
coboundary $\lambda\mapsto\tau_\lambda(h)-h$. Hence
\begin{equation}\label{eq:L418-group-comparison}
H^1(\Lambda,M)
\cong H^1(K^\bullet(M;\Delta)).
\end{equation}
The sign convention for coboundaries is immaterial since the expression
$b-\tau_\lambda(b)$ in the lemma is the coboundary associated to $-b$.

Now for each $1\le i\le r$, define a constant derivation $D_i$ of $M$ by
\[
D_i(P_v):=\ell(v)(e_i).
\]
Since the $D_i$ commute and lower polynomial degree, they are locally
nilpotent. By Taylor's formula, 
$\tau_{e_i}=\exp(-D_i)$.
Therefore
\begin{equation}\label{eq:L418-difference-factorization}
\Delta_i=D_iU_i=U_iD_i,
\qquad
U_i:=\sum_{n\ge0}\frac{(-1)^{n+1}}{(n+1)!}D_i^n
=-\operatorname{id}_M+\frac12D_i-\frac16D_i^2+\cdots.
\end{equation}
Every series here terminates on each polynomial.  
In particular, if $D_i=0$, then
$U_i=-\operatorname{id}_M$.

Since $U_i+\operatorname{id}_M$ lowers degree, it is locally
nilpotent. It follows that $U_i$ is invertible, with inverse
\[
U_i^{-1}
=-\sum_{n\ge0}(U_i+\operatorname{id}_M)^n.
\]
Again this sum is finite on every polynomial. All the operators
$D_i,U_j,U_j^{-1}$ commute.

There is an explicit isomorphism of cochain complexes
\[
F^\bullet:K^\bullet(M;\Delta)
\longrightarrow K^\bullet(M;D).
\]
For $I:=\{i_1<\cdots<i_q\}$, set
$\varepsilon_I:=\varepsilon_{i_1}\wedge\cdots
\wedge\varepsilon_{i_q}$ and define
\begin{equation}\label{eq:L418-chain-isomorphism}
F^q(m\varepsilon_I)
:=\left(\prod_{i\in I}U_i^{-1}\right)(m)\varepsilon_I.
\end{equation}
The empty product is the identity, so $F^0=\operatorname{id}_M$.
For $j\notin I$, commutativity gives
\[
\left(\prod_{i\in I\cup\{j\}}U_i^{-1}\right)\Delta_j
=D_j\left(\prod_{i\in I}U_i^{-1}\right).
\]
Comparing with \eqref{eq:L418-koszul-differential} proves
$F^{q+1}d_\Delta=d_DF^q$. Each $F^q$ is invertible, so
\eqref{eq:L418-chain-isomorphism} is indeed an isomorphism of complexes.
In degree one it sends
\[
(c_1,\ldots,c_r)
\longmapsto(U_1^{-1}c_1,\ldots,U_r^{-1}c_r).
\]

Now define the complex-linear map
\[
L:E\longrightarrow R^*,
\qquad L(\lambda\otimes1)(v):=\ell(v)(\lambda).
\]
Let $N:=\ker L$ and $s:=\operatorname{rank}L$. For $e\in E$, let
$D_e$ be the constant derivation satisfying 
$$D_e(P_v)=L(e)(v).$$
The map $e\mapsto D_e$ is linear and $D_{e_i}=D_i$.
Also,
\[
R_0=\{v\in R\mid L(e)(v)=0,\,\forall e\in E\}.
\]

Choose a complement $E'\subseteq E$ to $N$, and a basis
$e'_1,\ldots,e'_r$ of $E$ such that
\[
E'=\operatorname{span}\{e'_1,\ldots,e'_s\},
\qquad
N=\operatorname{span}\{e'_{s+1},\ldots,e'_r\}.
\]
The functionals $L(e'_1),\ldots,L(e'_s)$ are linearly independent.
Choose $v_1,\ldots,v_s\in R$ with
\[
L(e'_a)(v_b)=\delta_{ab}\qquad(1\le a,b\le s),
\]
and choose a basis $w_1,\ldots,w_{\dim R-s}$ of $R_0$.
Then the $v_a$ and $w_b$ together form a basis of $R$.
Set
\[
x_a:=P_{v_a},\qquad z_b:=P_{w_b},\qquad
B:=\operatorname{Sym}(R_0)=\mathbb C[z_1,\ldots,z_{\dim R-s}].
\]
Thus
\[
M=B[x_1,\ldots,x_s],
\qquad
D_{e'_a}=\frac{\partial}{\partial x_a}\ (1\le a\le s),
\qquad
D_{e'_a}=0\ (a>s).
\]
Changing basis in $E$ and simultaneously using its dual basis only
rewrites the same Koszul differential. This change is made only after
passing to the linear family $D_e$.

Since $\Delta_i=U_iD_i$ with $U_i$ invertible,
\[
M^\Lambda
=\bigcap_i\ker\Delta_i
=\bigcap_i\ker D_i
=\operatorname{Sym}(R_0).
\]
In particular, the elements of $B=\operatorname{Sym}(R_0)$ are fixed by every translation.

Next let $\varphi_1,\ldots,\varphi_r$ be dual to the basis
$e'_1,\ldots,e'_r$ of $E$. The formulas above identify
\[
K^\bullet(M;D)
\cong
\Omega^\bullet_{M/B}\otimes\wedge^\bullet N^*.
\]
Here $\varphi_a$ for $a\le s$ corresponds to $dx_a$, and
$\wedge^\bullet N^*$ has zero differential. The differential on
the tensor product is $d_{\mathrm{dR}}\otimes\operatorname{id}$,
with total degree equal to the sum of the two exterior degrees.
By \eqref{eq:L418-derham-cohomology},
\begin{equation}\label{eq:L418-koszul-cohomology}
H^1(K^\bullet(M;D))\cong B\otimes N^*.
\end{equation}
Concretely, every closed $1$-cochain $\mathbf a$ has the form
\begin{equation}\label{eq:L418-derivative-decomposition}
\mathbf a=d_Dh+\sum_{j=1}^{r-s}g_j\nu_j,
\qquad h\in M,\quad g_j\in B,
\end{equation}
where $\nu_1,\ldots,\nu_{r-s}\in E^*$ vanish on $E'$ and restrict
to a basis of $N^*$. To see this directly, write
$\mathbf a=\sum_{a=1}^r A_a\varphi_a$. Closedness gives
\[
\frac{\partial A_b}{\partial x_a}
=\frac{\partial A_a}{\partial x_b}\quad(a,b\le s),
\qquad
\frac{\partial A_b}{\partial x_a}=0\quad(a\le s<b).
\]
The first equations say that $\sum_{a\le s}A_a\,dx_a$ is a closed
relative polynomial $1$-form, hence is $d_{\mathrm{dR}}h$.
The second say that $A_b\in B$ for $b>s$.
This proves \eqref{eq:L418-derivative-decomposition}.

Now apply \eqref{eq:L418-derivative-decomposition} to
$\mathbf a:=F^1(\mathbf c)$. In the original basis of $E$, this says
\[
U_i^{-1}(c_i)
=D_i(h)+\sum_j g_j\nu_j(e_i).
\]
Since $g_j\in B$, all $D_i(g_j)$ vanish, hence
$U_i(g_j)=-g_j$. Multiplying the preceding equality by $U_i$ gives
\[
c_i=\Delta_i(h)-\sum_jg_j\nu_j(e_i).
\]
Set
$b:=-h$, $f_j:=g_j$, and 
$\alpha_j(\lambda):=-\nu_j(\lambda\otimes1)$.
Then $\alpha_j$ is an additive homomorphism $\Lambda\to\mathbb C$,
and
\[
c_{e_i}=b-\tau_{e_i}(b)+\sum_jf_j\alpha_j(e_i).
\]

Finally, define for every $\lambda\in\Lambda$
\[
\widetilde c_\lambda
:=b-\tau_\lambda(b)+\sum_jf_j\alpha_j(\lambda).
\]
Since each $f_j$ is fixed by the translations and each $\alpha_j$
is additive, $\widetilde c$ satisfies
\eqref{eq:L418-cocycle}. It agrees with $c$ on the basis elements
$e_i$. A group $1$-cocycle is determined by those values; its identity
implies
$c_0=0$, $c_{-\lambda}=-\tau_{-\lambda}(c_\lambda)$,
and then determines its values on all integer sums of the $e_i$.
Thus $c=\widetilde c$, proving \eqref{eq:L418-conclusion}.
The argument also covers zero-dimensional spaces and empty bases,
with the corresponding sums interpreted as zero.
\end{proof}

\section{General properties of one-sided twisted function algebras}
\label{sec:general-results}

In this section we establish the results on growth, simplicity, primitive ideals,
radicals, and finite-dimensional representations for arbitrary affine algebraic groups.

\subsection{Growth under Hopf \texorpdfstring{$2$}{2}-cocycle twisting}
\label{sec:prelim-growth}

The following observation that growth is twist-invariant applies to arbitrary finitely generated Hopf
algebras.

\begin{proposition}\label{prop:growth-general}
Let $A$ be a finitely generated Hopf algebra over $\C$, and let $J$ be a
Hopf $2$-cocycle for $A$.  There is a finite-dimensional subcoalgebra
$C\subseteq A$, which contains $1$ and generates $A$ as an algebra, such
that for every $n\geq0$, the subspace spanned by products of at most
$n$ elements of $C$ is the same for the original multiplication of $A$ and
for the multiplication of $A_J$.  Thus,
\[
\GKdim A_J=\GKdim A.
\]
\end{proposition}

\begin{proof}
Choose finitely many algebra generators of $A$.  By the fundamental
theorem of coalgebras, each generator belongs to a finite-dimensional
subcoalgebra of $A$.  Therefore the sum of these subcoalgebras and $\C1$ is a
finite-dimensional subcoalgebra $C\subseteq A$, which generates $A$.

Let $F_nA$ be the span of ordinary products of at most $n$ elements of
$C$, and let $F_n A_J$ be the analogous span for the multiplication of
$A_J$.  Since products of subcoalgebras are subcoalgebras, it follows that
\[
a\cdot_Jb
=
a_{(1)}b_{(1)}J(a_{(2)},b_{(2)})
\in F_{p+q}A,\qquad a\in F_pA,\,\,b\in F_qA.
\]
Thus, by induction, $F_nA_J\subseteq F_nA$ for every $n$.

Conversely, since
$ab=(a_{(1)}\cdot_Jb_{(1)})J^{-1}(a_{(2)},b_{(2)})$, $a,b\in A$,
induction yields $F_nA\subseteq F_nA_J$ for every $n$.  Thus, the two
filtrations coincide and have the same growth.
\end{proof}

\subsection{Minimality and simplicity}
\label{sec:prelim-simplicity}
We now answer the simplicity part of \cite[Question~6.1(1)]{G}. 

\begin{theorem}
\label{thm:minimal-simple}
Let $G$ be an affine algebraic group over $\C$, and let $J$ be a Hopf
$2$-cocycle for $\Os(G)$.  Then the following are equivalent:
\begin{enumerate}
\item $J$ is minimal.
\item $\Supp(J)=G$.
\item The one-sided twisted algebra $\Os(G)_J$ is simple.
\end{enumerate}
\end{theorem}

\begin{proof}
The equivalence of \textup{(1)} and \textup{(2)} is
\eqref{eq:minimal-support}, which follows from
\cite[Theorem~3.1]{G}.  

The coproduct of $\Os(G)$ defines a right $\Os(G)^J$-coaction on the algebra $\Os(G)_J$, and Hopf $2$-cocycle
twisting carries the right $\Os(G)$-Galois object $\Os(G)$ to the right
$\Os(G)^J$-Galois object $\Os(G)_J$. Namely, the canonical map
\begin{equation*}
\gamma:=\gamma_J:
\Os(G)_J\otimes \Os(G)_J\longrightarrow \Os(G)_J\otimes \Os(G)^J,
\qquad
a\otimes b\longmapsto(a\cdot_Jb_{(1)})\otimes b_{(2)},
\end{equation*}
is bijective. Moreover, since $\Os(G)$ is commutative, we have for every $a,b\in \Os(G)$,
\begin{align*}
(b_{(1)}\cdot_Ja_{(1)})R^J(a_{(2)},b_{(2)})
&=
a_{(1)}b_{(1)}J(b_{(2)},a_{(2)})
J^{-1}(b_{(3)},a_{(3)})J(a_{(4)},b_{(4)})
\notag\\
&=
a_{(1)}b_{(1)}J(a_{(2)},b_{(2)})
=
a\cdot_Jb.
\end{align*}
Thus, $\Os(G)_J$ is $R^J$-commutative. Now if $J$ is minimal, then $R^J$ is nondegenerate, and
Lemma~\ref{lem:braided-galois-simple} implies that $\Os(G)_J$ is simple, which
proves that (1) implies (3).

Conversely, suppose that $J$ is not minimal.  Then by
\cite[Theorem~3.1]{G}, there are a proper closed subgroup
$H:=\Supp(J)\subsetneq G$, and a minimal Hopf $2$-cocycle $J_H$ for
$\Os(H)$, such that $J$ is gauge equivalent to 
$J':=J_H\circ(\pi\otimes\pi)$,
where $\pi:\Os(G)\twoheadrightarrow\Os(H)$ is the canonical surjection. Then we have
$\Os(G)_J\cong\Os(G)_{J'}$ as algebras. For $f,g\in\Os(G)$ we have
\begin{align*}
 \pi(f\cdot_{J'}g)
 &=\pi(f_{(1)}g_{(1)})
       J_H\bigl(\pi(f_{(2)}),\pi(g_{(2)})\bigr)=\pi(f)\cdot_{J_H}\pi(g).
\end{align*}
Hence restriction induces an algebra surjection $\pi:\Os(G)_{J'}\twoheadrightarrow\Os(H)_{J_H}$.
Since $H\subsetneq G$, its kernel is nonzero and proper. Thus
$\Os(G)_{J'}$, hence $\Os(G)_J$, is not simple. This proves
that (3) implies (1).
\end{proof}

\begin{remark}
\label{rem:center-versus-simplicity}
Theorem~\ref{thm:minimal-simple} is stronger than 
$Z(\Os(G)_J)\cong\Os(G/\Supp(J))$ \cite[Theorem~3.5]{G}.  For a non-affine homogeneous space,
$\Os(G/\Supp(J))$ may equal $\C$ when
$\Supp(J)\neq G$, hence trivial center does not imply minimality for an
arbitrary affine algebraic group. \qed
\end{remark}

\subsection{Primitive and maximal ideals}\label{sec:Primitive and maximal ideals}
We now give an affirmative answer to \cite[Question 6.1(5)]{G}.

\begin{theorem}\label{thm:q615-main}
Let $G$ be an affine algebraic group, let $J$ be a Hopf
$2$-cocycle for $\mathscr O(G)$, and let $H$ be the support of $J$.  Then the following hold:
\begin{enumerate}
\item
$\operatorname{Prim}(\mathscr O(G)_J)=\operatorname{Max}(\mathscr O(G)_J)$
and this set is a single $G$-orbit.
\item
Choose a support representative
$J^{\mathrm s}:=J_H\circ(\pi\otimes\pi)$ and a $G$-equivariant algebra
isomorphism
\[
 \Phi:\Os(G)_J\xrightarrow{\cong}\Os(G)_{J^{\mathrm s}}.
\]
For $a\in G$, let
$\pi_a^{\mathrm s}(f)(h):=f(ah)$ and $P_{aH}:=\Phi^{-1}(\ker\pi_a^{\mathrm s})$ (when $J=J^{\mathrm s}$ and $\Phi=\id$, the underlying vector
subspace $P_{aH}$ is $I(aH)$).
Then the map
\begin{equation}\label{eq:q615-parametrization}
 G(\C)/H(\C)\xrightarrow{\cong}{\rm Prim}(\Os(G)_J)={\rm Max}(\Os(G)_J),
 \qquad aH\longmapsto P_{aH},
\end{equation}
is a $G$-equivariant homeomorphism for the Zariski topology on closed points
and the hull--kernel topology. Moreover, every primitive quotient is isomorphic to
$\Os(H)_{J_H}$. 
\end{enumerate}
\end{theorem}

\begin{proof}
Transporting through $\Phi$, it suffices to work with the support
representative $J=J_H\circ(\pi\otimes\pi)$. In the remainder of the proof
we suppress the superscript $\mathrm s$ and write
$\pi_a(f)(h)=f(ah)$ and $P_{aH}=\ker\pi_a$.
Restriction to the closed coset $aH$ is a surjective comodule algebra map
before twisting. Twisting gives the algebra surjection
$\pi_a:\Os(G)_J\twoheadrightarrow\Os(H)_{J_H}$ whose image is simple by Theorem \ref{thm:minimal-simple}.
Therefore, $P_{aH}$ is maximal and hence primitive.  Since distinct cosets
have distinct defining ideals, \eqref{eq:q615-parametrization} is
injective.

To prove that (\ref{eq:q615-parametrization}) is surjective, first assume that $G=G^\circ H$.  Choose the functions
$f_1,\ldots,f_s$ from Lemma~\ref{lem:q615-affine-charts}.  Since they have no
common zero in $G$, 
$\sum_i \mathscr O(G)f_i=\mathscr O(G)$ by the Nullstellensatz.  Hence by \eqref{eq:q615-normal-ideals}, we have
\begin{equation}\label{eq:q615-cover-ideals}
 \sum_i \mathscr O(G)_J\cdot_J f_i=\mathscr O(G)_J.
\end{equation}
Now let $M$ be a simple $\mathscr O(G)_J$-module.
By \eqref{eq:q615-cover-ideals}, some $f=f_i$ acts nontrivially on $M$.
Since $f$ is normal, both $fM$ and $\ker f$ are
$\mathscr O(G)_J$-submodules of $M$, so $f$ acts bijectively by simplicity of $M$.  Thus, $M$
extends to a simple $(\mathscr O(G)_J)[f^{-1}]$-module.

Since $(\mathscr O(G)_J)[f^{-1}]$ is countable-dimensional over $\mathbb C$ and $M$
is a simple, hence cyclic, $(\mathscr O(G)_J)[f^{-1}]$-module, $M$ is also
countable-dimensional. Hence by Dixmier's lemma
\cite[Lemma~0.5.2]{W},
\[
\operatorname{End}_{(\mathscr O(G)_J)[f^{-1}]}(M)
=
\mathbb C.
\]
Thus, the central subalgebra
$\mathscr O(X_f)\subseteq Z((\mathscr O(G)_J)[f^{-1}])$ from Lemma \ref{lem:q615-fiber}(1)
acts on $M$ through a character
\[
\mathscr O(X_f)\longrightarrow\mathbb C.
\]
Its kernel is $\mathfrak m_x$ for some $x:=aH\in X_f(\mathbb C)$.
Thus, $M$ is a module over
$(\mathscr O(G)_J)[f^{-1}]/\mathfrak m_x(\mathscr O(G)_J)[f^{-1}]$.  Since this quotient is simple by
\eqref{eq:q615-simple-fiber}, we deduce
\[
 \operatorname{Ann}_{(\mathscr O(G)_J)[f^{-1}]}(M)=\mathfrak m_x(\mathscr O(G)_J)[f^{-1}].
\]
Since the quotient map is restriction to $aH$, contracting its kernel to
$\mathscr O(G)_J$ gives
\[
\operatorname{Ann}_{\mathscr O(G)_J}(M)=\mathscr O(G)_J\cap\mathfrak m_x(\mathscr O(G)_J)[f^{-1}]=\ker\pi_a=P_{aH},
\]
which proves surjectivity of \eqref{eq:q615-parametrization} when $G=G^\circ H$.

For arbitrary $G$, let $L:=G^\circ H$; it is an open and closed
subgroup of $G$ of finite index. Since the characteristic functions of its left
cosets are right $H$-invariant central idempotents in $\mathscr O(G)_J$, we get a 
finite product decomposition
\[
\mathscr O(G)_J\cong\prod_{\gamma\in G/L}\mathscr O(\gamma L)_{J},
\]
whose factors are isomorphic to
$\mathscr O(L)_{J}$ by left translation.  Every simple $\mathscr O(G)_J$-module $M$ is supported on exactly one
factor.  Since $L=L^\circ H$, the preceding argument applies to this
factor and again gives $\operatorname{Ann}_{\mathscr O(G)_J}(M)=P_{aH}$ for some $a\in G$. Thus, \eqref{eq:q615-parametrization} is a bijection.
In particular, every primitive ideal of $\mathscr O(G)_J$ is maximal (and vice versa).

Finally, since for $a_0\in G$,
$a_0\cdot P_{aH}=P_{a_0aH}$, it follows that the bijection \eqref{eq:q615-parametrization} is $G$-equivariant, and its image is a single $G$-orbit.

For a two-sided ideal $I\subseteq \mathscr O(G)_J$, the inverse image in $G$ of
$\{aH\mid I\subseteq P_{aH}\}$
is the common zero set of all functions $a\mapsto f(ah)$ with
$f\in I$ and $h\in H$; it is closed and right $H$-stable.
Since the quotient map $G\to G/H$ is an open surjection, this defines a
closed subset of $G/H$.  Thus, the bijection is continuous.

Conversely, if $Y\subseteq G/H$ is closed, then
$I(q^{-1}(Y))\subseteq\mathscr O(G)$ is a right $H$-stable ideal.
It remains a two-sided ideal after twisting, and its hull in
$\operatorname{Prim}(\mathscr O(G)_J)$ corresponds exactly to $Y$, proving
continuity of the inverse.
\end{proof}

We can now give an affirmative answer to \cite[Question~6.1(4)]{G}.

\begin{corollary}\label{cor:q615-radical}
Let $G$ be an affine algebraic group, and let $J$ be a Hopf $2$-cocycle for $G$. Then
the algebra $\Os(G)_J$ is semiprime and semiprimitive; that is, 
$\Rad(\Os(G)_J)=0= \Jac(\Os(G)_J)$.
\end{corollary}

\begin{proof}
Let $H$ be the support of $J$. Then after the gauge replacement, we have by Theorem \ref{thm:q615-main},
\[
 \operatorname{Jac}(\mathscr O(G)_J)
 =\bigcap_{aH\in G/H}P_{aH}
 =\bigcap_{aH\in G/H}I(aH)=0
\]
since the $H$-cosets cover $G$ and $\mathscr O(G)$ is reduced.  Finally, since $\Rad(\Os(G)_J)\subseteq  \Jac(\Os(G)_J)$, it is also zero.
\end{proof}

\begin{corollary}\label{cor:q615-affine-center}
If $G/H$ is affine then, under the center identification of
\cite[Theorem~3.5]{G}, every primitive ideal of $\mathscr O(G)_J$ is generated by its
central character; that is,
\[
 P_{aH}=\mathfrak m_{aH}\cdot_J\mathscr O(G)_J,
 \qquad \mathfrak m_{aH}\in\operatorname{Max}(\mathscr O(G/H)).
\]
\end{corollary}

\begin{proof}
This follows from Lemma~\ref{lem:q615-fiber} for $f=1$ and 
$X_f=G/H$.
\end{proof}

\begin{remark}
For a non-affine quotient $G/H$, the global center need not distinguish
primitive ideals.  In 
\cite[Example~3.7]{G}, $G=\operatorname{SL}_2$ and $G/H\cong\mathbb P^1$,
whereas $Z(\mathscr O(G)_J)=\mathbb C$.  Nevertheless, Theorem~\ref{thm:q615-main}
parametrizes the primitive ideals of $\mathscr O(G)_J$ by $\mathbb P^1(\mathbb C)$. \qed
\end{remark}

\begin{corollary}\label{thm:fd-classification}
Let $G$ be an affine algebraic group, and let $J$ be a
Hopf $2$-cocycle for $\Os(G)$ with support $H$.
\begin{enumerate}
\item The following are equivalent:
\begin{enumerate}
\item
$\Os(G)_J$ has a nonzero finite-dimensional representation.
\item
$\Os(G)_J$ has a finite-dimensional irreducible representation.
\item
$H$ is finite.
\end{enumerate}
\item 
If $H$ is finite, then every simple $\Os(G)_J$-module is finite-dimensional of dimension $\sqrt{|H|}$, and  
${\rm Irr}_{{\rm fd}}(\Os(G)_J)\cong(G/H)(\mathbb C)$.
\end{enumerate}
\end{corollary}

\begin{proof}
A nonzero finite-dimensional
module has a simple quotient, so (a) implies (b). By Theorem \ref{thm:q615-main}, every
simple $\Os(G)_J$-module factors through a quotient isomorphic to the simple
algebra $\Os(H)_{J_H}$. A nonzero representation of $\Os(H)_{J_H}$ is faithful. Thus,
a finite-dimensional simple $\Os(G)_J$-module would embed $\Os(H)_{J_H}$ into a finite-dimensional
matrix algebra. Since $\Os(H)_{J_H}$ has underlying vector space $\Os(H)$, this
forces $H$ to be finite.

Conversely, if $H$ is finite, then $\Os(H)_{J_H}$ is a finite-dimensional simple
complex algebra, hence $\Os(H)_{J_H}\cong M_d(\C)$ with $d^2=|H|$.
Theorem \ref{thm:q615-main} gives one such quotient for every $aH\in(G/H)(\C)$ and
shows that every simple $\Os(G)_J$-module factors through one of them.
There is exactly one simple module up to isomorphism for each matrix-algebra
quotient, of dimension $d$. Distinct cosets have distinct annihilators.
This proves all assertions.
\end{proof}

\section{Hopf \texorpdfstring{$2$}{2}-cocycles for connected solvable groups}\label{sec:solvable-cocycles}

In this section $G$ is a {\bf connected solvable} affine algebraic group over
$\C$, $U:=R_u(G)$, $T:=G/U$,
and we choose a maximal torus, again denoted by $T$, so that
$G=T\ltimes U$. 
We also let
\[
   \mathfrak g:=\Lie(G),\qquad
   \mathfrak u:=\Lie(U),\qquad
   \Lambda:=X(T).
\]
Thus, $\Lambda\cong\mathbb Z^{\dim T}$ and
$\Os(T)=\C[\Lambda]$.

\begin{proposition}\label{prop:classical-fiber}
Every fiber functor $F:\Rep(G)\rightarrow\operatorname{Vect}$ is classical (i.e., dimension-preserving).  Hence,
every fiber functor is given by a Hopf $2$-cocycle for $\Os(G)$.
\end{proposition}

\begin{proof}
By the Lie-Kolchin theorem, every irreducible finite-dimensional
rational representation of $G$ is one-dimensional.  Thus every simple object
is invertible, so a tensor functor sends it to an invertible vector space,
which must also be one-dimensional.  Then exactness of $F$, and a composition
series, give
$\dim F(V)=\dim V$ 
for every $V\in\Rep(G)$.  Since $\Os(G)$ is pointed, the Galois object associated with a fiber functor is cleft by \cite[Theorem~1.9]{Bi}; equivalently, it is
represented by a Hopf $2$-cocycle.
\end{proof}

Let $J$ be a Hopf $2$-cocycle for $G$, and let $H:=\Supp(J)$.

\begin{corollary}\label{cor:center-solvable}
There exist $k,\ell\in\mathbb{Z}^{\ge 0}$ such that
$G/H\cong\mathbb A^k\times(\Gm)^{\ell}$
as varieties, and hence
\[
   Z(\Os(G)_J)
   \cong
   \C[x_1,\ldots,x_k,y_1^{\pm1},\ldots,y_{\ell}^{\pm1}]
\]
as algebras. In particular,
$J$ is minimal if and only if $Z(\Os(G)_J)=\C$.
\end{corollary}

\begin{proof}
A connected solvable group over $\C$ is split.  By Rosenlicht's theorem
\cite[Theorem 1]{B}, every homogeneous variety
under a split solvable group is a product of affine lines and punctured
affine lines.  Thus, $G/H$ is affine and has the stated form. Since by  
\cite[Theorem 3.5]{G} $Z(\Os(G)_J)=\Os(G/H)$, the stated algebra isomorphism follows.

Finally, since $G$ is
connected, $G/H$ is a point exactly when $H=G$.
\end{proof}

\subsection{The algebra \texorpdfstring{$\mathscr O(G)_J$}{O(G)\_J} as an iterated Ore extension}\label{sec:Ore extensions}

Choose a normal additive series 
$$U=U_0\supset\cdots\supset U_n=1$$
as in Lemma~\ref{lem:normal-series}, and let
\[
G_i:=G/U_i,\qquad A_i:=\Os(G_i),\qquad 0\le i\le n.
\]
Thus $G_0=T$, $G_n=G$, and the quotient maps give the Hopf-subalgebra
tower $A_0\subset\cdots\subset A_n=\Os(G)$. We use the coordinates
$z_i$, characters $\chi_i$, and tensors $w_i$ supplied by
Lemma~\ref{lem:triangular-coordinate}. For a Hopf $2$-cocycle $J$ on
$\Os(G)$, let $J_i$ denote its restriction to $A_i\otimes A_i$.

The coefficient quantum torus is written
$\C^\eta[\Lambda]$ as in \eqref{eq:algebraic-quantum-torus}.

\begin{theorem}\label{thm:solvable-ore-structure}
Let \(G\) be a connected solvable affine algebraic group, let
\(n:=\dim U\), and let \(J\) be a Hopf \(2\)-cocycle for
\(\Os(G)\).  Then there exist
$
\eta\in Z^2(\Lambda,\mathbb C^\times)
$,
and derivations \(\delta_i\) of $(A_{i-1})_{J_{i-1}}$, such
that
\begin{equation}\label{eq:solvable-ore}
\Os(G)_J
\cong
\mathbb C^\eta[\Lambda]
[z_1;\delta_1]\cdots[z_n;\delta_n]
\end{equation}
as algebras. Moreover, this presentation has a PBW filtration for which
\begin{equation}\label{eq:solvable-gr}
\operatorname{gr}\Os(G)_J
\cong
\mathbb C^\eta[\Lambda]\otimes
\mathbb C[z_1,\ldots,z_n]
\end{equation}
as algebras, where the variables on the right commute with the quantum torus and
with one another.
\end{theorem}

\begin{proof}
Let \(J_i\) be the restriction of \(J\) to \(A_i\otimes A_i\), and let
$B_i:=(A_i)_{J_i}$.
Also write
\[
w_i=\sum_s p_s\otimes q_s.
\]
Then for \(a\in A_{i-1}\), we have
$$
z_i\cdot_J a=
z_i a
+\chi_i a_{(1)}J_i(z_i,a_{(2)})
+\sum_s p_s a_{(1)}J_i(q_s,a_{(2)}),
$$
and
$$
a\cdot_J z_i=
a z_i
+a_{(1)}\chi_iJ_i(a_{(2)},z_i)
+\sum_s a_{(1)}p_sJ_i(a_{(2)},q_s).
$$
Since the leading terms cancel, and all remaining terms belong to
\(A_{i-1}\), we have
\[
\delta_i(a)
:=
z_i\cdot_Ja-a\cdot_Jz_i
\in B_{i-1}\qquad (a\in A_{i-1}).
\]
Since \(\delta_i\) is the restriction of the inner derivation
\(\operatorname{ad}(z_i)\) of \(B_i\), it is an ordinary derivation of
\(B_{i-1}\).

Now filter \(A_i=A_{i-1}[z_i]\) by ordinary \(z_i\)-degree.  
For $a,b\in A_{i-1}$ and $p,q\geq0$, the leading
$z_i$-degree term of
$(az_i^p)\cdot_{J_i}(bz_i^q)$ is
$(a\cdot_{J_{i-1}}b)z_i^{p+q}$.
Moreover, the symbol of $z_i$ commutes with the coefficient algebra.
Thus $\operatorname{gr}_{z_i}B_i\cong B_{i-1}[\overline z_i]$, and
$B_i\cong B_{i-1}[z_i;\delta_i]$ as algebras. 

Give the quantum-torus algebra degree zero. Inductively choose
positive integers $N_i=\deg z_i$. Fix a finite generating set
of $B_{i-1}$ consisting of torus generators, their inverses,
and the earlier Ore generators. Choose $N_i$ so large that
$\deg\delta_i(s)<\deg s+N_i$
for every member $s$ of this generating set. The derivation
rule gives the same estimate on filtered products. The resulting
weighted PBW filtration satisfies
\[
\operatorname{gr}B_n
\cong\mathbb C^\eta[\Lambda]
       \otimes\mathbb C[\overline z_1,\ldots,\overline z_n],
\qquad \deg\overline z_i=N_i.
\]

Finally, $A_0=\mathscr O(T)=\mathbb C[\Lambda]$, and since every \(X^\lambda\) is grouplike, we have
\[
X^\lambda\cdot_JX^\mu
=
J(X^\lambda,X^\mu)X^{\lambda+\mu}.
\]
Thus,
$
B_0\cong\mathbb C^\eta[\Lambda]$ and 
$
\eta(\lambda,\mu):=J(X^\lambda,X^\mu)$,
and iteration proves \eqref{eq:solvable-ore}-\eqref{eq:solvable-gr}.
\end{proof}

\begin{corollary}\label{cor:ring-properties}
For every connected solvable affine algebraic group $G$, and every Hopf
$2$-cocycle $J$ for $\Os(G)$, the following hold:
\begin{enumerate}
\item The algebra $\Os(G)_J$ is an affine left and right Noetherian domain.
\item The Gelfand--Kirillov dimension of $\Os(G)_J$ is $\dim G$.
\end{enumerate}
\end{corollary}

\begin{proof}
(1) The twisted group algebra of the free abelian group $\Lambda$ is an
iterated skew Laurent polynomial algebra, hence an affine Noetherian
domain, and differential Ore extensions preserve Noetherianity and the
domain property.  

(2) This follows from Proposition~\ref{prop:growth-general}.
\end{proof}  

\begin{remark}
For connected solvable groups, Corollary \ref{cor:ring-properties} answers the Noetherianity
question for $\Os(G)_J$ and the domain question in
\cite[Question~6.1(2)--(3)]{G}. \qed
\end{remark}

\begin{corollary}\label{cor:solvable-simple-properties}
If \(J\) is a minimal Hopf
$2$-cocycle for $G$, then \(\Os(G)_J\) is a simple affine
Noetherian domain, which is an iterated differential Ore
extension of a quantum torus, $Z(\Os(G)_J)=\mathbb C$,  
$$
\operatorname{GKdim} \Os(G)_J =\dim G,$$
and it has a PBW basis
$
\left\{
X^\lambda z_1^{a_1}\cdots z_n^{a_n}\mid
\lambda\in\Lambda,\, a_i\geq0
\right\}$. \qed
\end{corollary}

\begin{example}\label{ex:basic-solvable-blocks}

(1) If \(G=T\) is a torus, then minimality is equivalent to
nondegeneracy of the corresponding alternating bicharacter
$q(\lambda,\mu)=\eta(\lambda,\mu)\eta(\mu,\lambda)^{-1}$, and in this case 
$\mathscr O(T)_J\cong\mathbb C^\eta[\Lambda]$
is a simple quantum torus.  

(2)  If \(G=U\) is unipotent and \(J\) is minimal, then
\(\dim U=2m\), and $\mathscr O(U)_J\cong \mathcal A_m(\mathbb C)$ as algebras by \cite[Theorem 4.7]{G}. 
Thus, the unipotent simple algebras are precisely Weyl algebras.

(3)  If \(G=T\times U\) is connected nilpotent, let
\(V\subseteq U\) be the support of the restriction of \(J\) to
\(\mathscr O(U)\), and write \(\dim V=2m\).  Then by
\cite[Theorem 4.12]{G},
$\Os(G)_J\cong
\bigl(\mathscr O(U/V)\otimes \mathcal A_m(\mathbb C)\bigr)
\#_\sigma\mathbb C[\Lambda]$,
where the character lattice acts by triangular translation
automorphisms.  For minimal \(J\), this crossed product is simple.

(4) Let $G:=\mathbb G_m\ltimes\mathbb G_a$, so that 
$
\mathscr O(G)=\mathbb C[x^{\pm1},y]$ with
$\Delta(x)=x\otimes x$ and $\Delta(y)=y\otimes1+x\otimes y$.
Let
\[
X:=x\frac{\partial}{\partial x},
\qquad
Y:=x\frac{\partial}{\partial y}.
\]
Then $[X,Y]=Y$. For $h\in\mathbb C^\times$, set
\[
\mathcal F_h
:=\exp\bigl(-X\otimes\log(1-hY)\bigr)
=\sum_{n\geq0}\frac{h^n}{n!}
 X(X+1)\cdots(X+n-1)\otimes Y^n,
\]
where the product for $n=0$ is $1$. Let $J_h:=J_{\mathcal F_h}$
be the Hopf $2$-cocycle defined by \eqref{eq:J-from-F}; explicitly,
\[
J_h(f,g)
=(\varepsilon\otimes\varepsilon)
 \bigl(\mathcal F_h\triangleright(f\otimes g)\bigr).
\]
By direct computation, 
$
x\cdot_Jy-y\cdot_Jx=h x^2$, so  
setting
$
E:=-h^{-1}x^{-1}\cdot_Jy
$
yields
\[
E\cdot_Jx-x\cdot_JE=x.
\]
Therefore, we have algebra isomorphisms 
\[
\mathscr O(G)_{J_h}
\cong
\mathbb C[x^{\pm1}]
\left[E;x\frac{d}{dx}\right]
\cong
\D(\mathbb G_m)
\cong
\mathcal A_1(\mathbb C)[x^{-1}].
\]
Since this algebra is simple, $J_h$ is minimal by Theorem \ref{thm:minimal-simple}. 
In particular, a simple $\Os(G)_J$ for solvable $G$ need not be a Weyl algebra. \qed
\end{example}

\begin{corollary}\label{cor:solvable-dimension-two}
Let \(G\) be a connected solvable affine algebraic group of dimension \(2\), and let \(J\) be a
minimal Hopf $2$-cocycle for $G$.  Then \(\Os(G)_J\) is isomorphic to exactly one of
$
\mathcal A_1(\mathbb C)$, $
\D(\mathbb G_m)$, or
$
\mathbb C\langle x^{\pm1},y^{\pm1}\rangle/(xy-qyx)$ for a non-root of unity $q\in\mathbb C^\times$.
Moreover, \(q\) and \(q^{-1}\) define isomorphic quantum tori,
and these are the only identifications within that family.
\end{corollary}

\begin{proof}
Let
$r:=\dim T$ and $n:=\dim U$ with $r+n=2$
(the cases \(r=0\) and \(r=2\) are respectively the unipotent and
torus cases).

Suppose \(r=n=1\).  Since
$
H^2(\mathbb Z,\mathbb C^\times)=0$, it follows from 
Theorem~\ref{thm:solvable-ore-structure} that
$\Os(G)_J
\cong
\mathbb C[x^{\pm1}][z;\delta]$
for a derivation
$\delta=f(x)\frac{d}{dx}$. 
By simplicity, \(\mathbb C[x^{\pm1}]\) is
\(\delta\)-simple. The case $f=0$ gives a central polynomial variable and a nonzero proper
ideal, so it is excluded by simplicity.

If \(f\ne 0\) is a nonunit, then the proper
ideal \((f)\) is \(\delta\)-stable, since
$
\delta(f)=ff'$.
Thus, $f(x)=c x^m$ for some 
$c\neq0$ and $m\in\mathbb Z$.
Letting
$
E:=c^{-1}x^{1-m}z
$
gives $Ex-xE=x$,
so \(\Os(G)_J\cong\D(\mathbb G_m)\).

Let $\mathcal T_q:=\C\langle x^{\pm1},y^{\pm1}\rangle/
(xy-qyx)$. Its units are exactly the scalar monomials $cx^ay^b$.
Indeed, a translation-invariant ordering of $\mathbb Z^2$ shows that a finite
support with more than one term cannot have a multiplicative inverse with
finite support. Thus an isomorphism
$\mathcal T_{q'}\xrightarrow{\cong}\mathcal T_q$ sends its two standard generators to
$cx^ay^b$ and $ex^ry^s$, where
$\left(\begin{smallmatrix}a&r\\b&s\end{smallmatrix}\right)\in\mathrm{GL}_2(\mathbb Z)$.
Since their commutation factor is $q^{as-br}$, we have
$q'=q^{\pm1}$. Conversely, interchanging the two generators realizes
$q\leftrightarrow q^{-1}$. This proves the parameter assertion.

Finally, these three types are pairwise nonisomorphic. Their groups of
units modulo $\C^\times$ are respectively $0$, $\mathbb Z$, and $\mathbb Z^2$ for
$\mathcal A_1(\C)$, $\mathcal D(\Gm)$, and $\mathcal T_q$.
For the first two assertions, the standard polynomial filtration on the
Weyl algebra and the differential-order filtration on $\mathcal D(\Gm)$ show
that a unit has degree, respectively order, zero. The quantum-torus
assertion was proved above.
\end{proof}

\subsection{Minimal Hopf \texorpdfstring{$2$}{2}-cocycles for \texorpdfstring{$G$}{G}}
Let $T:=G/U$ and choose a splitting
$i:T\hookrightarrow G$, with quotient map $q:G\twoheadrightarrow T$.
Then
\[
 q^*:\Os(T)\hookrightarrow\Os(G),\qquad
 i^*:\Os(G)\twoheadrightarrow\Os(T),\qquad i^*q^*=\id.
\]
Restriction of a Hopf $2$-cocycle to the toral \emph{subalgebra} uses $q^*$ and inflation
of a toral Hopf $2$-cocycle to $G$ uses $i^*$.

Let $\mathcal J(G)$ denote the set of gauge-equivalence classes of Hopf
$2$-cocycles for $\mathscr O(G)$. By Proposition \ref{prop:classical-fiber}, 
restriction to the
canonical semisimple tensor subcategory
$\Rep(T)\subseteq \operatorname{Rep}(G)$ yields a map
\begin{equation}\label{eq:res-ss}
\operatorname{res}_{\mathrm{ss}}:\mathcal J(G)
\longrightarrow
\mathcal J(T)\cong H^2(\Lambda,\mathbb C^\times)
\cong
\operatorname{Hom}(\wedge^2\Lambda,\mathbb C^\times).
\end{equation}
Set
\begin{equation}\label{eq:J0G}
\mathcal J_0(G):=\mathrm{res}_{\mathrm{ss}}^{-1}([\varepsilon\otimes\varepsilon]),
\end{equation}
so that elements of $\mathcal J_0(G)$ correspond to the fiber functors which are standard on the
semisimple tensor subcategory $\Rep(T)\subseteq \Rep(G)$.

Define
\begin{equation}\label{eq:R0g}
\mathcal R_0(\mathfrak g)
:=
\left\{
 r\in\mathfrak g\wedge\mathfrak u\mid 
 \operatorname{CYB}(r)=0\right\}.
\end{equation}
Fix a choice of the universal quantization construction used in
\cite[Section~3]{EG1}, expressed in the right Hopf $2$-cocycle
convention of \S\ref{sec:prelim-cocycles}. For $r\in\mathcal R_0(\mathfrak g)$, write
$J_f(r)$ for the resulting Hopf $2$-cocycle for $G$; the subscript $f$ records
this fixed choice. Its algebraic definition follows from nilpotency since
on every tensor product of finite-dimensional rational representations,
the universal series terminates.

\begin{theorem}\label{thm:solvable-torus-trivial}
There is a natural bijection $\mathcal J_0(G)
\cong\mathcal R_0(\mathfrak g)/G$,
where $G$ acts by the adjoint action.  Moreover, the orbit of $r\in g\wedge\mathfrak u$ corresponds to
the universal quantization Hopf $2$-cocycle $J_f(r)$ for $G$.
\end{theorem}

\begin{proof}
Since a class belongs to $\mathcal J_0(G)$ precisely when its fiber functor on $\Rep(T)=\operatorname{Rep}(G/U)$
is standard, it is a unipotent fiber
functor in the sense of \cite{EG1}.  Applying the classification of
unipotent fiber functors to the symmetric tensor category
$\operatorname{Rep}(G)$ gives a skew solution
$r\in\wedge^2\mathfrak g$ of the classical Yang--Baxter equation,
unique up to $G$-conjugacy.  The nilpotence condition in that theorem is
exactly the vanishing of the image of $r$ in
$(\mathfrak g/\mathfrak u)^{\otimes2}$, hence
$r\in\mathfrak g\wedge\mathfrak u$.  

Conversely, the universal
quantization formula for such an $r$ defines a Hopf $2$-cocycle
$J_f(r)$ for $G$, which is trivial on $\Rep(T)=\operatorname{Rep}(G/U)$.
\end{proof}

For a skew solution $r$ of the CYBE, its {\bf carrier}
$\mathfrak q_r:=\im\bigl(r^\sharp:\mathfrak g^*\to\mathfrak g\bigr)$ 
is a Lie subalgebra, and $r$ is nondegenerate on
$\mathfrak q_r$.  Let $Q_r\subseteq G$ be the {\bf smallest closed connected} affine algebraic
subgroup whose Lie algebra contains $\mathfrak q_r$; equivalently,
$\Lie(Q_r)$ is the {\bf algebraic hull} of $\mathfrak q_r$ in $\mathfrak g$.

\begin{proposition}
\label{prop:support-hull}
For every $r\in\mathcal R_0(\mathfrak g)$, the support of
$J_f(r)$ is conjugate to $Q_r$.  Thus,
$J_f(r)$ is minimal for $G$ if and only if $Q_r=G$.
\end{proposition}

\begin{proof}
The universal quantization formula uses only the carrier's enveloping algebra and factors through $\Os(Q_r)$.
By Lemma \ref{lem:characters}, the unipotent radical of $Q_r$ is $Q_r\cap U$. Since
$r\in\wedge^2\Lie(Q_r)$ and its projection to
$\wedge^2(\mathfrak g/\mathfrak u)$ vanishes, we have
\[
 r\in\Lie(Q_r)\wedge\Lie(Q_r\cap U).
\]
Thus the same nilpotent universal quantization defines a cocycle on
$Q_r$ itself, and its inflation is $J_f(r)$ on $G$.
Applying the support theorem \cite[Theorem 3.1]{G} within $Q_r$, choose a
closed subgroup $S\subseteq Q_r$ and a minimal Hopf $2$-cocycle $K$ on $\Os(S)$ whose
inflation to $G$ is gauge equivalent to $J_f(r)$.

We claim that $K$ has trivial restriction to the diagonalizable quotient
$D:=S/(S\cap U)$. Its toral commutator bicharacter pulls back, along the
surjection $X(G)\twoheadrightarrow X(S)=X(D)$ of Lemma \ref{lem:characters}, to the toral
commutator bicharacter of $J_f(r)$, which is trivial. Hence the bicharacter
of the restriction of $K$ is trivial. Since $D$ is diagonalizable and
$\C^\times$ is divisible, the gauge class of a Hopf $2$-cocycle on $\Os(D)$ is
classified by its alternating bicharacter. The restriction of $K$ is
therefore gauge trivial.

The fiber functor corresponding to $K$ on $\Rep(S)$ is thus
unipotent in the sense of \cite[Definition~3.1]{EG1}. By
\cite[Theorem~3.2]{EG1}, it is represented by $J_f(s)$ for
\[
 s\in\Lie(S)\wedge\Lie(S\cap U),\qquad\CYB(s)=0.
\]
Naturality of the universal quantization under $S\hookrightarrow G$ shows
that the inflation of this Hopf $2$-cocycle is $J_f(s)$ on $G$. Uniqueness in the same
theorem, now applied to $G$, implies that $s$ is $G$-conjugate to $r$.
The algebraic hull of the carrier of $s$ is contained in $S$, so
$\dim Q_r\leq\dim S$. Since $S\subseteq Q_r$ and $Q_r$ is connected, it
follows that $S=Q_r$. This proves the support assertion up to conjugacy.
\end{proof}

\begin{corollary}\label{cor:quasi-frobenius-hull}
If $J_f(r)$ is minimal, then its carrier $\mathfrak q_r$ is an
even-dimensional quasi-Frobenius Lie algebra whose algebraic hull is
$\mathfrak g$.  If $\mathfrak q_r$ is itself algebraic, then
$\mathfrak q_r=\mathfrak g$,
so $G$ admits a left-invariant algebraic symplectic form, and $\dim G$ is
even.  In particular, this recovers \cite[Theorem 5.1]{G} when $G$ is
unipotent, since every Lie subalgebra of a unipotent algebraic group is
algebraic. \qed
\end{corollary}

\begin{example}
The even-dimensional conclusion does not extend to arbitrary solvable
or even nilpotent groups.  The failure is exactly the possible gap
between a Lie algebra and its algebraic hull.

For example, let $G=\Gm^2\times\Ga$. Then $\mathfrak g=\C X\oplus\C Y\oplus\C Z$ and $\mathfrak u=\C Z$,
with $X,Y$ corresponding to an integral basis of the cocharacter lattice.
Choose $\alpha\in\C\setminus\mathbb Q$, and let
$r:=(X+\alpha Y)\wedge Z$.
The Lie algebra is abelian, so $\CYB(r)=0$, 
$r\in\mathfrak g\wedge\mathfrak u$, and its carrier is the
symplectic Lie algebra
\[
   \mathfrak q_r:=\C(X+\alpha Y)\oplus\C Z.
\]
The toral line $\C(X+\alpha Y)$ is not rational with respect to the
cocharacter lattice, so its algebraic hull is all of
$\C X\oplus\C Y$.  Thus $Q_r=G$. By Proposition~\ref{prop:support-hull},
$J_f(r)$ is minimal, although $\dim G=3$. \qed
\end{example}

\subsection{Invariant Hopf \texorpdfstring{$2$}{2}-cocycles and the centrally split sector}

Let $Z=Z_r\times Z_u$ be the {\bf center} of $G$, where $Z_r$ is reductive and $Z_u$ is unipotent.  Let
$\mathfrak z_r:=\Lie(Z_r)$ and $\mathfrak z_u:=\Lie(Z_u)$.

A Hopf $2$-cocycle $L$ is invariant if $L*m=m*L$. Two invariant
Hopf $2$-cocycles are {\bf strongly} gauge equivalent if they are related by
\[
 L^x=(x\circ m)*L*(x^{-1}\otimes x^{-1})
\]
for a normalized convolution-invertible cochain $x\in\Os(G)^*$ that
is central in the convolution algebra $\Os(G)^*$. We write
$H^2_{\mathrm{inv}}(G)$ for the group of strong gauge classes of
invariant Hopf $2$-cocycles, as in \cite[Section~2]{EG2}. Ordinary gauge equivalence
allows arbitrary normalized convolution-invertible $x$ and is the
equivalence relation defining $\mathcal J(G)$.

For a normalized invertible cochain $x\in\Os(G)^*$, we use the right-convention
gauge formula
\[
 J^x=(x\circ m)*J*(x^{-1}\otimes x^{-1}).
\]
If $L$ is invariant, then $L$ commutes with $x\circ m$, hence
$L*J^x=(L*J)^x$.

\begin{theorem}
\label{thm:invariant-solvable}
There is a natural group isomorphism
\[
   H^2_{\mathrm{inv}}(G)
   \cong
   \Hom(\wedge^2\widehat{Z_r},\C^\times)
   \times(\mathfrak z_r\otimes\mathfrak z_u)
   \times(\wedge^2\mathfrak u)^G.
\]
\end{theorem}

\begin{proof}
This is the specialization of \cite[Theorem 7.8]{EG2} to a connected
solvable affine algebraic group, for which the unipotent radical of $G$ is $U$.
\end{proof}

Call
$[\sigma]\in H^2(\Lambda,\mathbb C^\times)$ {\bf centrally supported} if its
alternating bicharacter is supported on $Z_r$.  For a centrally supported toral class $[\sigma]$, let
$A_\sigma\subseteq Z_r$ be its diagonalizable support. Choose $L_\sigma$
to be the inflation of a cocycle on $\Os(A_\sigma)$ representing this toral
class. In particular, $L_\sigma$ has no mixed or unipotent component in the
decomposition of Theorem \ref{thm:invariant-solvable}. We call this a {\bf pure toral lift}.

\begin{theorem}\label{thm:centrally-supported-fibers}
For every centrally supported $[\sigma]\in H^2(\Lambda,\mathbb C^\times)$, multiplication by $L_\sigma$
gives a bijection
\[
\mathcal R_0(\mathfrak g)/G
\xrightarrow{\cong}
\operatorname{res}_{\mathrm{ss}}^{-1}([\sigma]),
\qquad
[r]\longmapsto[L_\sigma*J_f(r)].
\]

Moreover,
$
\operatorname{Supp}(L_\sigma*J_f(r))
=
A_\sigma Q_r$ (up to conjugacy),
so $
L_\sigma*J_f(r)$ is minimal if and only if 
$A_\sigma Q_r=G$.
\end{theorem}

\begin{proof}
By \cite[Lemma 2.1]{EG2}, invariant Hopf $2$-cocycles act by convolution on arbitrary Hopf $2$-cocycles.  Thus, if $J$ restricts to $[\sigma]$, then
$L_\sigma^{-1}*J$ is a torus-trivial Hopf $2$-cocycle and hence gauge
equivalent to $J_f(r)$ for a unique $G$-orbit of $r$.  

Choose a support subgroup $S\subseteq A_\sigma Q_r$ for $
L_\sigma*J_f(r)$. Since its toral restriction factors through the image of $S$ in $T$, $q(A_\sigma)\subseteq q(S)$. A diagonalizable Levi subgroup of
$S$ maps isomorphically to $q(S)$ and can be conjugated into $T$ inside the
connected solvable group $G$. Since $A_\sigma$ is central, this implies
$A_\sigma\subseteq S$ before as well as after the conjugation.
Thus $L_\sigma^{-1}$ is a Hopf $2$-cocycle supported on $S$. Removing it within $S$
shows that a Hopf $2$-cocycle gauge equivalent to $J_f(r)$ is supported on $S$.
By Proposition \ref{prop:support-hull}, a conjugate of $Q_r$ is contained in $S$.
Centrality of $A_\sigma$ now gives a conjugate of $A_\sigma Q_r$ contained in $S$.
Together with $S\subseteq A_\sigma Q_r$, equality follows since the dimensions are equal,
and the two copies of $A_\sigma Q_r$ have the same finite number of connected components.
Hence $S=A_\sigma Q_r$, up to the support conjugacy already chosen.
\end{proof}

For an arbitrary toral class \([\sigma]\in H^2(\Lambda,\C^\times)\), inflate it to a Hopf \(2\)-cocycle
\(\widetilde\sigma\) for \(G\).

\begin{proposition}\label{prop:relative-toral-reduction}
Fix a normalized cocycle $\sigma$ on $\Os(T)$ and let
$\widetilde\sigma:=\sigma\circ(i^*\otimes i^*)$.
There is a bijection between the normalized Hopf $2$-cocycles $J$ on $\Os(G)$ with
$J|_{q^*\Os(T)^{\otimes2}}=\sigma$ and the
normalized Hopf $2$-cocycles $K$ on $\Os(G)^{\widetilde\sigma}$ with trivial restriction to the same
toral subalgebra. The maps are
\[
 K=\widetilde\sigma^{-1}*J,\qquad J=\widetilde\sigma*K.
\]
They satisfy
\[
 \left(\Os(G)^{\widetilde\sigma}\right)^K=\Os(G)^J,
 \qquad R^J=K_{21}^{-1}*R^{\widetilde\sigma}*K.
\]
The bijection respects gauge morphisms whose restrictions preserve the fixed
toral representative. Every cocycle whose toral gauge class is $[\sigma]$
admits such a normalized representative.
\end{proposition}

\begin{proof}
This is the standard composition law for Hopf \(2\)-cocycle
twistings.  After choosing a representative of the gauge class of
\(J\), whose restriction to \(\Os(T)\) is \(\sigma\), twisting
first by \(\widetilde\sigma\) and then by
\(K=\widetilde\sigma^{-1}*J\) gives the same deformation as twisting by \(J\).  The restriction of \(K\) to \(\Os(T)\) is trivial,
and the displayed formula for the cotriangular structure is the usual
twisting formula.
\end{proof}

\begin{question}
\label{q:noncentral-torus}
For \(G=T\ltimes U\) and $
[\sigma]\in H^2(\Lambda,\C^\times)$,
whose support is not central in \(G\), classify the Hopf
\(2\)-cocycles \(K\) for \(\Os(G)^{\widetilde\sigma}\) whose restriction to
\(\Os(T)\) is trivial.  For which of them is
$
K_{21}^{-1}*R^{\widetilde\sigma}*K$ nondegenerate? \qed
\end{question}

Choose the {\bf maximal} torus $T$ so that $Z_r\subseteq T$.  Let
\[
   \res_T:H^2_{\mathrm{inv}}(G)\longrightarrow\Jset(T)
\]
be restriction.  The image of $\res_T$ consists precisely of toral cocycle classes
whose support is contained in the central diagonalizable group $Z_r$.

Define the {\bf centrally split sector}
\begin{equation}\label{eq:centrally-split-sector}
   \Jcs(G)
   :=\left\{[J]\in\Jset(G)\mid
      [J|_{\Os(T)^{\otimes 2}}]\in\im(\res_T)\right\},
\end{equation}
and set $K_0:=\ker(\res_T)$.

Recall that by Theorem~\ref{thm:solvable-torus-trivial}, $J_0(G)$ 
is the set of $G$-orbits of elements
$r\in\mathcal R_0(\mathfrak g)$.

\begin{theorem}
\label{thm:central-split}
Multiplication of Hopf $2$-cocycles induces a surjective map
\[
   H^2_{\mathrm{inv}}(G)\times\mathcal J_0(G)
   \longrightarrow\Jcs(G),
   \qquad
   ([L],[K])\longmapsto[L*K],
\]
whose fibers are the orbits of $K_0$ acting by
$[M]\cdot([L],[K])
   :=([L*M^{-1}],[M*K])$.
Thus,
\begin{equation}\label{eq:central-split-quotient}
   \Jcs(G)
   \cong
   \bigl(H^2_{\mathrm{inv}}(G)\times\mathcal J_0(G)\bigr)/K_0.
\end{equation}
\end{theorem}

\begin{proof}
The map is well defined on the stated classes. Indeed,
$L*K^x=(L*K)^x$ for an invariant $L$ and an arbitrary invertible
normalized cochain $x$. If $z$ is central in $\Os(G)^*$, then
$L^z*K=(L*K)^z$,
since $z^{-1}\otimes z^{-1}$ commutes with $K$. Thus changing $L$
within its strong gauge class or $K$ within its ordinary gauge class
does not change the ordinary gauge class of $L*K$.

Let $[J]\in\mathcal J_{\mathrm{cs}}(G)$. By the definition of this
set, choose an invariant Hopf $2$-cocycle $L$ whose restriction to the
toral subalgebra is gauge equivalent to the restriction of $J$.
After changing $J$ by a gauge transformation inflated from $T$, while
keeping $L$ fixed, we may assume that the two restrictions are equal.
By \cite[Lemma~2.1]{EG2}, invariant Hopf $2$-cocycles act by convolution
on arbitrary Hopf $2$-cocycles. Thus,
$K:=L^{-1}*J$
is a Hopf $2$-cocycle with trivial restriction to the toral subalgebra.
Thus $[K]\in\mathcal J_0(G)$ and $[J]=[L*K]$.

Conversely, if $L$ is invariant and $K\in\mathcal J_0(G)$, then $L*K$ is a
Hopf $2$-cocycle for $G$, and its toral restriction lies in $\im(\res_T)$.
This proves surjectivity.

Finally, if $L*K$ and $L'*K'$ are gauge equivalent, then
$M:=L'^{-1}*L$ is invariant and
$[K']=[M*K]$.
Since both $K$ and $K'$ have trivial toral restriction, $M\in K_0$, as desired.
\end{proof}

\begin{corollary}\label{cor:nilpotent-recovery}
If $G=T\times U$ is a connected nilpotent affine algebraic group, then $T$ is central and
$\Jcs(G)=\Jset(G)$. Thus, Theorem~\ref{thm:central-split} recovers \cite[Theorem 5.3]{G}.  For a non-nilpotent
solvable group, the possible non-central support of the toral restriction
is the new obstruction. \qed
\end{corollary}

Proposition~\ref{prop:relative-toral-reduction} reduces the part not covered by
Theorem~\ref{thm:central-split} to torus-trivial cocycles on the generally
noncommutative cotriangular Hopf algebra \(\Os(G)^{\widetilde\sigma}\).  The resulting
noncentral-torus problem is formulated as Question~\ref{q:noncentral-torus},
and we now discuss a possible color-Lie approach to it. (See, e.g. \cite{Scheunert} for color Lie algebras, \cite[Section~3]{CSV} for cocycle 
twisting of color Lie algebras, \cite{K} for its categorical
generalization to Lie algebra objects in module and comodule
categories, \cite{HM1} for color Lie bialgebras, their cohomology, and their deformation theory,
and \cite{HM2} for the color classical Yang--Baxter
equation and formal quantization of triangular color Lie bialgebras.)

Choose a bicharacter representative \(\sigma\), and let
$q_\sigma(\alpha,\beta):=\sigma(\alpha,\beta)\sigma^{-1}(\beta,\alpha)$.
With respect to the \(T\)-weight decomposition
$\mathfrak u=\bigoplus_{\alpha\in \Lambda}\mathfrak u_\alpha$,
define for homogeneous elements,
\[
[x_\alpha,y_\beta]_\sigma
:=
\sigma(\alpha,\beta)[x_\alpha,y_\beta],
\]
and denote the resulting {\bf color} Lie algebra by
\(\mathfrak g^\sigma\).
Give $\mathfrak t$ degree zero, so that
$\mathfrak g=\mathfrak t\oplus\mathfrak u$ is $\Lambda$-graded.
Also call an element $r_\sigma\in\wedge^2_{q_\sigma}\mathfrak g^\sigma$
{\bf relative} if
$(p\otimes p)(r_\sigma)=0$, where
$p:\mathfrak g^\sigma\twoheadrightarrow
\mathfrak g^\sigma/\mathfrak u^\sigma$.
For example, the ordinary
counterpart is the nilpotence condition of
\cite[Definition~2.13]{EG1}, used in the classification of
unipotent fiber functors in \cite[Theorem~3.2]{EG1}.

The remaining problem is Question \ref{q:noncentral-torus}. We discuss a possible color-CYBE
approach, without asserting a classification or an algebraic quantization
in this generality. Proposition \ref{prop:relative-toral-reduction} reduces the remaining
problem to Hopf $2$-cocycles on $\Os(G)^{\widetilde\sigma}$ with trivial toral
restriction. Cocycle twisting of graded Lie algebras and formal quantization
of color Lie bialgebras suggest a relative Lie-theoretic description
\cite{CSV,K,HM2}. To formulate such a description, one must specify the
allowed graded or relative tensors, their equivalence relation, the
compatibility with change of the chosen torus, and an algebraic
quantization whose support is defined. The condition of zero projection
to the toral tensor square is motivated by
\cite[Definition~2.13 and Theorem~3.2]{EG1}, but alone does not provide these
additional structures. We leave this as a direction for Question \ref{q:noncentral-torus}.

\subsection{Minimal Hopf \texorpdfstring{$2$}{2}-cocycles via multi-extended Jordanian twists}\label{sec:extended}

We now use \S\ref{sec:jordanian-twists} to give a family for which the one-sided twisted algebra
can be computed completely.

Fix an integer $d\neq0$, and integers $p_i,q_i\in\mathbb Z$, $1\leq i\leq m$, such that
\[
 p_i+q_i=d,
   \qquad 1\leq i\leq m.
\]
Set $\mathbf p:=(p_1,\dots,p_m)$ and $\mathbf q:=(q_1,\dots,q_m)$. Let
\[
   G(d;\mathbf p,\mathbf q)
   :=\Gm\ltimes H_{2m+1}
\]
be the affine algebraic group whose underlying variety is
$\Gm\times\mathbb A^{2m+1}$
with coordinates
\[
   t^{\pm1},a_1,\ldots,a_m,b_1,\ldots,b_m,c,
\]
and multiplication rule
\begin{align}
 &(t,\mathbf a,\mathbf b,c)
  (t',\mathbf a',\mathbf b',c')
=
 \left(
 tt',
 \mathbf a+t^{\mathbf p}\cdot\mathbf a',
 \mathbf b+t^{\mathbf q}\cdot\mathbf b',
 c+t^dc'+\sum_{i=1}^m a_i t^{q_i}b_i'
 \right)
 \label{eq:extended-group-law}
\end{align}
(with the obvious notation).
Thus, we have 
\begin{equation}\label{eq:coproduct-t}
   \Delta(t)=t\otimes t,
 \end{equation}
 \begin{equation}\label{eq:coproduct-a}  
   \Delta(a_i)=a_i\otimes1+t^{p_i}\otimes a_i,\qquad 1\leq i\leq m,
 \end{equation}
  \begin{equation}\label{eq:coproduct-b}
   \Delta(b_i)=b_i\otimes1+t^{q_i}\otimes b_i,\qquad 1\leq i\leq m,
  \end{equation}
   \begin{equation}\label{eq:coproduct-c}
   \Delta(c)=c\otimes1+t^d\otimes c
      +\sum_{i=1}^m a_i t^{q_i}\otimes b_i.
  \end{equation}
The relevant left-invariant vector fields are
\begin{equation}\label{eq:extended-vector-fields}
   H:=t\frac{\partial}{\partial t},\qquad
   A_i:=t^{p_i}\frac{\partial}{\partial a_i},\qquad
   B_i:=t^{q_i}
       \left(\frac{\partial}{\partial b_i}
       +a_i\frac{\partial}{\partial c}\right)\qquad (1\leq i\leq m),\qquad
   E:=t^d\frac{\partial}{\partial c},
\end{equation}
and they satisfy
\begin{equation}\label{eq:extended-brackets-one}
   [H,E]=dE,\qquad 
   [H,A_i]=p_iA_i\qquad (1\leq i\leq m),
 \end{equation} 
 \begin{equation}\label{eq:extended-brackets-two} 
   [H,B_i]=q_iB_i,
   \qquad
   [A_i,B_j]=\delta_{ij}E\qquad (1\leq i,j\leq m),
 \end{equation}
with all remaining brackets among the $A_i,B_i,E$ equal to zero.
Let
\[
   H_0:=d^{-1}H,
   \qquad
   \alpha_i:=\frac{p_i}{d},
   \qquad
   \beta_i:=\frac{q_i}{d}\qquad (1\leq i\leq m).
\]
Then $\alpha_i+\beta_i=1$ for every $1\leq i\leq m$.

For $\xi\in\C^\times$, set $\sigma_\xi:=\log(1+\xi E)$,
and consider the {\bf normalized multi-extended Jordanian twist}
\begin{equation}\label{eq:multi-extended-twist}
   \mathcal F_{\mathrm{EJ}}(\xi)
   :=
   \exp(-H_0\otimes\sigma_{-\xi})
   \prod_{i=1}^{m}\exp\bigl(\xi A_i\otimes B_i e^{-\beta_i\sigma_{-\xi}}\bigr)
\end{equation}
(see \cite{KLM}, and compare also the chain formulation in
\cite{KLO}). 
Note that the factors in the product commute. 
Each series acts locally finitely on
$\Os(G)^{\otimes 2}$ since the operators $E,A_i,B_i$ are locally nilpotent
on regular functions, while every positive power of $\sigma_{-\xi}$ has a
locally nilpotent operator in the second tensor factor.  Hence,
\eqref{eq:J-from-F} defines a Hopf $2$-cocycle
$J_{\mathrm{EJ}}(\xi)$ for $\Os(G(d;\mathbf p,\mathbf q))$.

\begin{theorem}\label{thm:extended-block}
With the notation above, there exist algebra isomorphisms 
\begin{equation}\label{eq:extended-isomorphism}
   \Os\bigl(G(d;\mathbf p,\mathbf q)\bigr)_{J_{\mathrm{EJ}}(\xi)}
   \cong
   \D(\Gm)\otimes \mathcal A_m(\C)
   \cong
   \D(\Gm\times\mathbb A^m).
\end{equation}
In particular, $J_{\mathrm{EJ}}(\xi)$ is minimal.  Remarkably, the
isomorphism type of the algebra is independent of the weights
$p_i,q_i\in \mathbb{Z}$ and of $\xi\in \mathbb{C}^{\times}$.
\end{theorem}

\begin{proof}
Set $G:=G(d;\mathbf p,\mathbf q)$ and $J:=J_{\mathrm{EJ}}(\xi)$.

Let $s:=\sigma_{-\xi}=\log(1-\xi E)$ and
$D_i:=B_i e^{-\beta_i s}$. Thus
\[
 \mathcal F_{\mathrm{EJ}}(\xi)
 =\exp(-H_0\otimes s)
   \prod_{i=1}^m\exp(\xi A_i\otimes D_i).
\]
The extension factors in the product over $i$ commute. Their
position to the right of the Jordanian factor is retained. On the coordinate
generators we have
\[
 s(c)=-\xi t^d,\qquad
 s(t)=s(a_j)=s(b_j)=0,
\]
\[
 e^{-\beta_i s}(c)=c+\beta_i\xi t^d,\qquad
 D_i(b_j)=\delta_{ij}t^{q_i},\qquad
 D_i(c)=a_i t^{q_i},\qquad
 D_i(t)=D_i(a_j)=0.
\]
Also $A_i(a_j)=\delta_{ij}t^{p_i}$ and $H_0(t)=d^{-1}t$.
Each extension correction has second tensor factor $t^{q_i}$ or
$a_i t^{q_i}$, which is annihilated by $E$. The outer Jordanian
factor therefore adds no further correction to those terms. Then  \eqref{eq:star-from-F}
gives
\begin{align*}
 t\cdot_Jc&=tc+\frac{\xi}{d}t^{d+1},
 &c\cdot_Jt&=ct,\\
 a_i\cdot_Jb_j&=a_ib_j+\delta_{ij}\xi t^d,
 &b_j\cdot_Ja_i&=b_ja_i,\\
 a_i\cdot_Jc&=a_ic+\xi a_it^d,
 &c\cdot_Ja_i&=ca_i.
\end{align*}
Here products on the right are ordinary products of coordinate functions.
Using these identities, we obtain
\begin{equation}\label{eq:relation-tc}
   [t,c]_J=\frac{\xi}{d}t^{d+1},
   \qquad
   [a_i,b_j]_J=\delta_{ij}\xi t^d,
   \qquad
   [a_i,c]_J=\xi a_it^d\qquad (1\leq i,j\leq m).
\end{equation}
All other commutators among $t,a_1,\ldots,a_m,b_1,\ldots,b_m,c$ vanish.
Relations involving $t^{-1}$ follow from
$[t^{-1},x]_J=-t^{-1}[t,x]_Jt^{-1}$.

Now define
\begin{equation}\label{eq:new-extended-generators}
   v_i:=\xi^{-1}t^{-d}b_i
   \qquad (1\leq i\leq m),\qquad
   P:=-d\xi^{-1}t^{-d}c
      +d\sum_{i=1}^m a_iv_i.
\end{equation}
Then relations \eqref{eq:relation-tc} imply
\begin{equation}\label{eq:canonical-relations}
   [P,t]_J=t,
   \qquad
   [a_i,v_j]_J=\delta_{ij}\qquad (1\leq i,j\leq m).
\end{equation}
Again all other commutators among $P,t,a_1,\ldots,a_m,v_1,\ldots,v_m$
vanish and relations involving $t^{-1}$ are obtained from
$[t^{-1},x]_J=-t^{-1}[t,x]_Jt^{-1}$; in particular,
\[
 [P,t^k]_J=kt^k\qquad(k\in\mathbb Z).
\]
For completeness, the cancellations in the mixed commutators are
\begin{equation*}
   \left[-d\xi^{-1}t^{-d}c,a_j\right]_J=d a_j,\qquad
   \left[d\sum_{i=1}^m a_iv_i,a_j\right]_J=-d a_j\qquad (1\leq j\leq m),
\end{equation*}
 \begin{equation*}  
   \left[-d\xi^{-1}t^{-d}c,v_j\right]_J=-d v_j,\qquad   
   \left[d\sum_{i=1}^m a_iv_i,v_j\right]_J=d v_j\qquad (1\leq j\leq m).
\end{equation*}

The relations define a unital homomorphism
\[
 \mathcal D(\Gm)\otimes \mathcal A_m(\C)\longrightarrow\Os(G)_J
\]
sending the standard generators to $t^{\pm1},P,a_i,v_i$.
The identities
\[
 b_i=\xi t^dv_i,\qquad
 c=-\frac{\xi}{d}t^d\left(P-d\sum_{i=1}^m a_iv_i\right)
\]
show that it is surjective. Since its source is the simple ring
$\mathcal D(\Gm\times\mathbb A^m)$, its kernel is zero. This proves (\ref{eq:extended-isomorphism}).
Theorem \ref{thm:minimal-simple} now implies that $J_{\mathrm{EJ}}(\xi)$ is minimal.
\end{proof}

Diatta-Manga-Mbaye prove that the \(2\)-solvable Frobenius Lie
algebras considered in \cite{DMM} split as
$\mathfrak g=\mathfrak b\ltimes V$,
$\dim\mathfrak b=\dim V=n$,
where \(\mathfrak b\) is a maximal abelian subalgebra of
\(\End(V)\).  A Frobenius functional determines a nondegenerate
triangular tensor \(r\in\wedge^2\mathfrak g\). Then Example~\ref{ex:basic-solvable-blocks} and 
Theorem~\ref{thm:extended-block} motivate the following.  

\begin{conjecture}
\label{conj:polarization}
Let $\mathfrak g:=\mathfrak b\ltimes V$ be a $2$-solvable Frobenius
Lie algebra as above, with $\mathfrak b$ and $V$ abelian and
$\dim\mathfrak b=\dim V=n$. Assume that
$\mathfrak b\subseteq\operatorname{End}(V)$ is algebraic, let
$B\subseteq\operatorname{GL}(V)$ be its connected algebraic group,
and set $G:=B\ltimes V_{\mathrm a}$. For a Frobenius functional
$\lambda\in\mathfrak g^*$, let
\[
 \omega_\lambda(x,y):=\lambda([x,y]),\qquad
 r_\lambda:=\omega_\lambda^{-1},\qquad
 J_\lambda:=J_f(r_\lambda).
\]
Then
$\Os(G)_{J_\lambda}\cong\mathcal D(B)
$
as algebras. In particular, if
$B\cong\Gm^s\times\Ga^{n-s}$, then
$$\Os(G)_{J_\lambda}
 \cong\mathcal D(\Gm)^{\otimes s}\otimes \mathcal A_{n-s}(\C)$$ as algebras. \qed
\end{conjecture}

Since $\mathfrak b$ and $V$ are abelian, the Frobenius form
$\omega_\lambda$ vanishes on each of these two summands.
Its nondegeneracy and their equal dimensions imply that
$r_\lambda\in\mathfrak b\wedge V$. Since
$V\subseteq\operatorname{Lie}(R_u(G))$, the tensor $r_\lambda$
belongs to $\mathcal R_0(\mathfrak g)$. Thus $J_f(r_\lambda)$
is defined by the construction preceding Theorem~\ref{thm:solvable-torus-trivial}.

The low-dimensional
classifications in \cite{DMM} provide a finite list of further test
cases for Conjecture~\ref{conj:polarization}.

\begin{example}\label{ex:named-families}
The following are special cases of Theorem~\ref{thm:extended-block}.
\begin{enumerate}
\item For $m=0$, the basic Jordanian carrier gives
$\Os(\Gm\ltimes\Ga)_J\cong\D(\Gm)$.

\item In the restricted Borel carrier of $\mathfrak{sl}_N$ from
\cite{KLM}, one has $m=N-2$, and after the customary normalization,
$d=2$ and $p_i=q_i=1$.  Hence,
$\Os(G)_J\cong\D(\Gm)\otimes \mathcal A_{N-2}(\C)$.

\item At a peripheric boundary, one may take
$p_i=0$, $q_i=d$ (or the reversed choice).  Then the simple algebra is still
$\D(\Gm)\otimes \mathcal A_m(\C)$,
although the carrier Lie algebra has changed.  This agrees with the
fact that the peripheric twists of \cite{LD1} remain Frobenius. \qed
\end{enumerate}
\end{example}

By Proposition~\ref{prop:external-products}, Theorem~\ref{thm:extended-block}
also produces arbitrarily large examples by taking external products.
Thus, combining Theorem~\ref{thm:extended-block} with a
nondegenerate toral bicharacter yields the simple algebras
\begin{equation}\label{eq:external-product-family}
   \D(\Gm)^{\otimes r}
   \otimes \mathcal A_m(\C)
   \otimes\C^\eta[\mathbb Z^{2s}].
\end{equation}

The chain constructions in \cite{KLO,LD2} use products of the
Jordanian and extension factors
\eqref{eq:Jordanian-factor}-\eqref{eq:extension-factor}, often
together with Reshetikhin factors.  Their carrier Lie algebras are
solvable, and for full chains, Frobenius.

\begin{conjecture}
\label{conj:chain}
Let \(G\) be a connected solvable affine algebraic carrier of a full chain of
extended Jordanian twists.  Suppose the chain has \(r\) Jordanian
links and a total of \(m\) extension pairs, so that
$\dim G=2r+2m$,
and assume that the chain parameters are nonzero and that the full carrier
$r$-matrix is nondegenerate. Take the right-oriented operator associated
to the standard chain as in \S\ref{sec:jordanian-twists}, and assume that it and its inverse
act on $\Os(G)^{\otimes2}$ as mutually inverse algebraic operators.  If
\(J_{\mathcal F}\) is the resulting Hopf \(2\)-cocycle for $G$, then after a
triangular algebraic change of generators, there is an algebra isomorphism
\begin{equation*}
\Os(G)_{J_{\mathcal F}}
\cong
\D(\Gm^r\times\mathbb A^m)
\cong
\D(\Gm)^{\otimes r}\otimes \mathcal A_m(\C).
\end{equation*}
If compatible Reshetikhin factors are added, then the resulting algebra is
a jointly nondegenerate Jordanian quantum differential algebra
$ A(\Lambda,\eta;V,\ell,\omega)$ (see Definition \ref{def:QDO} below). \qed
\end{conjecture}

Conjecture \ref{conj:chain} is proved for a single multi-extended link by
Theorem~\ref{thm:extended-block}, and for external products of such
links by Proposition~\ref{prop:external-products}.  The recursive
constructions in \cite{KLM,KLO,LD2} suggest that each Jordanian link
contributes one logarithmic coordinate, and its differential operator,
while each extension pair contributes one Weyl pair.

\begin{example}\label{ex:mixed-torus-jordanian}
Theorem~\ref{thm:extended-block} yields differential operators on a
commutative affine algebraic group.  A different phenomenon occurs when a
toral bicharacter and a logarithmic derivation interact.

In \cite[Example~4.13]{G}, a minimal Hopf $2$-cocycle on
$G:=\Gm^2\times\Ga$ 
produces the algebra generated by $x^{\pm1},y^{\pm1},z$ with relations
\begin{equation}\label{eq:G1-mixed-example}
   xy=qyx,
   \qquad
   yz=zy,
   \qquad
   xz=(z+1)x,
\end{equation}
where $q=e^h$ has infinite multiplicative order.  Equivalently,
$[z,x]=-x$ and $[z,y]=0$,
so this is a differential Ore extension of the simple quantum torus
$\C_q[x^{\pm1},y^{\pm1}]$ by the degree derivation
$\delta(x)=-x$ and $\delta(y)=0$.
Thus, even a commutative solvable affine algebraic group may yield a simple algebra in
which a quantum-torus block and a Jordanian logarithmic derivative are nontrivially coupled. \qed
\end{example}

Example~\ref{ex:basic-solvable-blocks},
Theorem~\ref{thm:extended-block}, and
Example~\ref{ex:mixed-torus-jordanian} motivate the following.

\begin{question}
\label{q:simple-solvable-twists}
Which simple algebras arise as \(\Os(G)_J\)? \qed
\end{question}

\section{A candidate normal family}\label{sec:normal-family}

The examples in \S\ref{sec:solvable-cocycles} are all contained in
the family $\mathscr A$ of algebras $A(\Lambda,\eta;V,\ell,\omega)$
introduced in Definition~\ref{def:QDO}. We now study its abelian
realization, its role in the proposed normal forms, and its derivations.

\begin{example}
  (1) For $V=0$, we get quantum tori.
  
   (2) For $\Lambda=0$, we get Weyl algebras when $\omega$ is nondegenerate.
   
   (3) For $\eta=1$, $\omega=0$, and 
    a perfect pairing $\ell$, we get $\D(\Gm^r)$.
   
   (4) Theorem~\ref{thm:extended-block} gives $\D(\Gm)\otimes \mathcal A_m(\C)$.
   
   (5) Example~\ref{ex:mixed-torus-jordanian} gives a coupled quantum-torus/degree-derivation algebra. \qed
\end{example}

We now realize the algebras $A(\Lambda,\eta;V,\ell,\omega)$ by abelian
Hopf $2$-cocycles and relate their simplicity to the joint radicals
\eqref{eq:KV}--\eqref{eq:KLambda}.

\begin{theorem}
\label{thm:abelian-realization}
Let $(\Lambda,\eta;V,\ell,\omega)$ be as in Definition \ref{def:QDO}.
Let $T$ be the algebraic torus with character lattice $\Lambda$,
and let $V^*_{\mathrm a}$ denote the vector group whose coordinate
space of linear functions is $V$.  Let
$G_0:=T\times V^*_{\mathrm a}$.
Then there is an explicit Hopf $2$-cocycle $J$ for
$\Os(G_0)$, such that
\begin{equation}\label{eq:abelian-realization}
   \Os(G_0)_J
   \cong
    A(\Lambda,\eta;V,\ell,\omega)
\end{equation}
as algebras. Furthermore, the following are equivalent:
\begin{enumerate}
\item $J$ is minimal.
\item $A(\Lambda,\eta;V,\ell,\omega)$ is simple.
\item $A(\Lambda,\eta;V,\ell,\omega)$ 
is jointly nondegenerate.
\end{enumerate}
Thus, every jointly nondegenerate algebra in $\mathscr A$
actually occurs as $\Os(G)_J$ for a commutative, hence connected
solvable affine algebraic group, $G$ and a minimal Hopf $2$-cocycle $J$ for $G$.
\end{theorem}

\begin{proof}
Let $\mathfrak t:=\Lie(T)=\Hom_{\mathbb Z}(\Lambda,\C)$.
Since $\Lambda$ is free abelian and the exponential map
$\C\to\C^\times$ is surjective, we can choose
$r_T\in\wedge^2\mathfrak t$ such that
\begin{equation}\label{eq:rT-logarithm}
   \exp\bigl(r_T(\lambda,\mu)\bigr)
   =q_\eta(\lambda,\mu)
   \qquad(\lambda,\mu\in\Lambda).
\end{equation}
Choose a basis $v_1,\ldots,v_n$ of $V$ and let
$v^1,\ldots,v^n$ be the dual basis of $V^*$, regard $\omega$ as the
corresponding element
$r_\omega\in\wedge^2V^*$, and set
\begin{equation}\label{eq:abelian-r-matrix}
   r
   :=
   r_T-
   \sum_{i=1}^n\ell(v_i)\wedge v^i
   +r_\omega
   \in
   \wedge^2(\mathfrak t\oplus V^*).
\end{equation}
Note that (\ref{eq:abelian-r-matrix}) is independent of the choice of basis of $V$.  Since
$\mathfrak t\oplus V^*$ is abelian, $r$ satisfies the classical
Yang--Baxter equation and
$\mathcal F_r:=\exp(r/2)$
is an abelian Drinfeld twist.  Since on a Laurent monomial in the torus
coordinates the toral part of the exponential acts by a scalar and
all terms involving a vector-group derivative are locally finite, \eqref{eq:J-from-F} defines an honest Hopf
$2$-cocycle $J_r$ for $\Os(G_0)$.

Now for $v\in V$, let $P_v$ be the corresponding linear coordinate on
$V^*_{\mathrm a}$.  Then a direct application of
\eqref{eq:star-from-F} yields
\begin{equation}\label{eq:abelian-star-torus}
   X^\lambda\cdot_{J_r}X^\mu=
   \exp\!\left(\frac12r_T(\lambda,\mu)\right)
   X^{\lambda+\mu},
 \end{equation}
 \begin{equation} \label{eq:abelian-star-mixed}
   [P_v,X^\lambda]_{J_r}
=\ell(v)(\lambda)X^\lambda,
  \end{equation}
   \begin{equation} \label{eq:abelian-star-additive}
   [P_v,P_w]_{J_r}
   =\omega(v,w).
 \end{equation}
Since the group cocycle in \eqref{eq:abelian-star-torus} has commutator
bicharacter $q_\eta$ (\ref{eq:rT-logarithm}), and since two cocycles on the free abelian group
$\Lambda$ with the same commutator bicharacter are cohomologous, it follows that after rescaling the generators $X^\lambda$ relations
\eqref{eq:abelian-star-torus}-\eqref{eq:abelian-star-additive} are precisely
\eqref{eq:QDO-torus}--\eqref{eq:QDO-additive}. 
Ordered twisted monomials have nonzero scalar multiples of the ordinary
Laurent-polynomial monomials as their leading terms in the total
vector-group degree. Terms involving vector-group derivatives lower that
degree. Thus the change between the two ordered monomial families is
triangular and invertible. The PBW basis therefore proves that the
resulting homomorphism is an isomorphism, which proves
\eqref{eq:abelian-realization}.

Finally, by the realization already proved, $A(\Lambda,\eta;V,\ell,\omega)\cong\Os(G_0)_{J_r}$ with $G_0$
connected and commutative. Lemma \ref{lem:center} shows that $Z(A)=\C$ exactly
when $K_V=0=K_\Lambda$. Corollary \ref{cor:center-solvable} identifies this condition with minimality
of $J_r$, and Theorem \ref{thm:minimal-simple} identifies minimality with simplicity. This proves
the equivalences of (1), (2), and (3).
\end{proof}

\subsection{Localization in the final algebra and quotient division rings}
\label{sec:localization-obstruction}

The most direct localization formulation does not yield a weaker
normal form, since the final algebra is simple.

\begin{proposition}
\label{prop:normal-localization-collapse}
Let \(G\) be a connected solvable affine algebraic group, let \(J\) be a minimal Hopf
\(2\)-cocycle for \(G\), and let \(S\) be an Ore set generated by
finitely many nonzero normal semi-invariants of
\(\Os(G)_J\).  Then $S^{-1}\Os(G)_J=\Os(G)_J$.
Thus, 
$S^{-1}\Os(G)_J\cong
 A(\Lambda,\eta;V,\ell,\omega)$
as algebras if and only if $\Os(G)_J\cong
 A(\Lambda,\eta;V,\ell,\omega)$
as algebras.
\end{proposition}

\begin{proof}
By Theorem~\ref{thm:minimal-simple}, \(\Os(G)_J\) is simple, so the claim follows from 
Lemma~\ref{lem:normal-localization-trivial}.
\end{proof}

\begin{remark}\label{rem:meaning-normal}
Proposition~\ref{prop:normal-localization-collapse} only shows that the denominators
cannot be required to be normal in the {\bf final} simple algebra.
They may be normal in successive coefficient algebras in the Ore tower
\eqref{eq:solvable-ore} and cease to be normal after later variables are
adjoined.  This is exactly what happens in the Weyl algebra;
\(x\) is normal in \(\C[x]\) but not in \(\mathcal A_1(\C)\). \qed
\end{remark}

Equality of quotient division rings does not imply an algebra isomorphism,
nor does it imply that the two algebras are related by localization at
normal elements of the final simple algebra. The following example
illustrates this distinction.

\begin{example}\label{ex:weyl-DGm}
Let $\mathcal A_1(\C)=\C\langle x,\partial\rangle/
(\partial x-x\partial-1)$.
Then $\mathcal A_1(\C)[x^{-1}]=\D(\Gm)$, so
\[
\operatorname{Frac}(\mathcal A_1(\C))
\cong
\operatorname{Frac}(\D(\Gm)).
\]
Both algebras belong to the family \(\mathscr A\) since 
$\mathcal A_1(\C)\cong
A(0,1;\C^2,0,\omega_{\mathrm{std}})$
and
\[
\D(\Gm)
\cong
A(\mathbb Z,1;\C,\ell,0),
\qquad
\ell(1)(1)=1.
\]
However they are not isomorphic (e.g. 
$\mathcal A_1(\C)^\times=\C^\times$ and 
$\D(\Gm)^\times\supseteq\C^\times x^{\mathbb Z}$).
Moreover, \(x\) is not normal in \(\mathcal A_1(\C)\) since otherwise
\(\partial x=x\partial+1\in xA_1(\C)\) would imply
\(1\in x\mathcal A_1(\C)\), contrary to the fact that \(x\) is not a unit. \qed
\end{example}

\begin{conjecture}
\label{conj:strong-normal-form}
Let $G$ be a connected solvable affine algebraic group and let \(J\) be a minimal Hopf \(2\)-cocycle for \(\Os(G)\).  Then the following hold:
\begin{enumerate}
\item
There are jointly nondegenerate data 
$(\Lambda,\eta;V,\ell,\omega)$, with
$\operatorname{rank}\Lambda+\dim V=\dim G$,
such that
$\Os(G)_J\cong A(\Lambda,\eta;V,\ell,\omega)$.
\item
There are jointly nondegenerate data 
$(\Lambda,\eta;V,\ell,\omega)$, with
$\operatorname{rank}\Lambda+\dim V=\dim G$,
such that $\operatorname{Frac}\bigl(\Os(G)_J\bigr)
\cong
\operatorname{Frac}\bigl(
 A(\Lambda,\eta;V,\ell,\omega)
\bigr)$, as algebras. \qed
\end{enumerate}
\end{conjecture}

\begin{remark}
Theorem~\ref{thm:abelian-realization} gives an explicit realization
map from the parameters of the conjectural classification to a
Hopf $2$-cocycle; namely,
\[
   (\eta,\ell,\omega)
   \longmapsto
   J_{\exp(r/2)},
   \qquad
   r:=r_T-\sum_i\ell(v_i)\wedge v^i+r_\omega.
\]
Thus, Theorem~\ref{thm:abelian-realization} proves that every algebra
in $\mathscr A$ occurs.  Conjecture~\ref{conj:strong-normal-form} is about the question of whether these exhaust the algebra isomorphism types, respectively quotient division rings,
of one-sided twists arising from minimal Hopf $2$-cocycles on connected solvable groups. \qed
\end{remark}

\subsection{Inductive birational deletion}

We now prove the strongest general statement obtainable from the Ore
presentation alone once the required outer-derivation calculation is
available.

We use the canonical derivations $\partial_{\alpha,\beta}$ of
\eqref{eq:canonical-derivation}, whose extension to quotient division
rings is established in Lemma~\ref{lem:canonical-derivation}.

\begin{theorem}
\label{thm:deletion-criterion}
Let $A_0:=\C^\eta[\Lambda]$ and $A_i:=A_{i-1}[z_i;\delta_i]$ $(1\leq i\leq n)$
be an iterated differential Ore extension of domains, and let
$D_i:=\operatorname{Frac} (A_i)$.  Start with
$V_0=0$, $A_0:=\C^\eta[\Lambda]$, and an algebra isomorphism 
$\Phi_0:D_0\xrightarrow{\cong}\operatorname{Frac} (A_0)$.
Suppose that, whenever $V_{i-1}$, $\ell_{i-1}$, $\omega_{i-1}$, and an algebra isomorphism
\[
   \Phi_{i-1}:D_{i-1}\xrightarrow{\cong}
   \operatorname{Frac}(A(\Lambda,\eta;V_{i-1},\ell_{i-1},\omega_{i-1}))
\]
have been constructed, we have
\begin{equation}\label{eq:outer-decomposition}
   \Phi_{i-1}\delta_i\Phi_{i-1}^{-1}
   =
   \ad(a_i)+\partial_{\alpha_i,\beta_i}
\end{equation}
for some $a_i\in
   \operatorname{Frac}(A(\Lambda,\eta;V_{i-1},\ell_{i-1},\omega_{i-1}))$, $\alpha_i\in\Hom_{\mathbb Z}(\Lambda,\C)$,
and $\beta_i\in V_{i-1}^*$.
Then there exists an algebra isomorphism
\begin{equation}\label{eq:deletion-conclusion}
   \Phi_i:\operatorname{Frac} (A_i)
   \xrightarrow{\cong}
   \operatorname{Frac}(A(\Lambda,\eta;V_i,\ell_i,\omega_i)),
\end{equation}
where
$V_i:=V_{i-1}\oplus\C v_i$
and the data are obtained recursively from
\begin{equation}\label{eq:extend-ell}
   \ell_i|_{V_{i-1}}=\ell_{i-1},\qquad
   \ell_i(v_i)=\alpha_i,   
\end{equation}
\begin{equation}\label{eq:extend-omega}   
   \omega_i|_{\wedge^2V_{i-1}}=\omega_{i-1},\qquad
   \omega_i(v_i,v)=\beta_i(v)
   \quad(v\in V_{i-1}).
\end{equation} 
In particular, there is an algebra isomorphism
\begin{equation}\label{eq:deletion-conclusionn}
   \operatorname{Frac} (A_n)
   \xrightarrow{\cong}
   \operatorname{Frac}(A(\Lambda,\eta;V_n,\ell_n,\omega_n)).
\end{equation}
\end{theorem}

\begin{proof}
Assume that $D_{i-1}$ has already been identified, as an algebra, with the quotient
division ring
\[
   {\rm Frac}(
   A(\Lambda,\eta;V_{i-1},\ell_{i-1},\omega_{i-1})),
\]
such that (\ref{eq:outer-decomposition}) holds.
Inside the differential Ore extension of $D_{i-1}$, after transporting by $\Phi_{i-1}$, identify the coefficient division rings and set
$P_{v_i}:=z_i-a_i$.
Then for $b\in D_{i-1}$, \eqref{eq:outer-decomposition} yields
\begin{align*}
   [P_{v_i},b]
   &=[z_i,b]-[a_i,b]=\partial_{\alpha_i,\beta_i}(b).
\end{align*}
In particular,
\[
   [P_{v_i},X^\lambda]
   =\alpha_i(\lambda)X^\lambda,
\qquad [P_{v_i},P_v]
   =\beta_i(v)
   \qquad(v\in V_{i-1}).
\]
Since these are exactly the defining relations of
$A_i:=A(\Lambda,\eta;V_i,\ell_i,\omega_i)$ 
with the extensions \eqref{eq:extend-ell}-\eqref{eq:extend-omega},
the PBW theorem for differential Ore extensions identifies
$D_{i-1}(z_i;\delta_i)\cong\operatorname{Frac} (A_i)$
as algebras. Thus, \eqref{eq:deletion-conclusion} follows by induction on $i$.
\end{proof}

\begin{corollary}
\label{cor:integral-deletion}
In the setting of Theorem~\ref{thm:deletion-criterion}, assume first that
at every stage, $\Phi_{i-1}$ identifies $A_{i-1}$ itself with
$A(\Lambda,\eta;V_{i-1},\ell_{i-1},\omega_{i-1})$ and $a_i$ in \eqref{eq:outer-decomposition} belongs to this
algebra (not only to its quotient division ring).  Then there is an algebra isomorphism
$A_n\cong A(\Lambda,\eta;V_n,\ell_n,\omega_n)$.

More generally, suppose that at every stage there is a localization
$B_{i-1}$ of $A_{i-1}$ at a finitely generated Ore set of
semi-invariants, such that
\begin{enumerate}
\item the localization is preserved by the remaining derivations,
\item $B_{i-1}$ is identified with a Jordanian quantum differential
algebra in $\mathscr A$,
\item $a_i\in B_{i-1}$, and
\item any further inversions are absorbed into a new
$\mathscr A$-presentation, for example by replacing an additive coordinate
which has been inverted by a lattice variable.
\end{enumerate}
Then the compatible localized tower has an algebra presentation in the
family $\mathcal A$, and hence the corresponding quotient division rings
are isomorphic.

In either case, if the original final algebra is $A_n=\Os(G)_J$ for
$G$ connected solvable and $J$ minimal, the resulting data are jointly
nondegenerate. Accordingly these hypotheses prove, respectively, the
strong and birational assertions of Conjecture \ref{conj:strong-normal-form} for that pair $(G,J)$.
\end{corollary}

\begin{proof}
Under the first set of assumptions, the change
$P_{v_i}:=z_i-a_i$ takes place inside the algebra at every stage, so the
proof of Theorem~\ref{thm:deletion-criterion} gives an algebra
isomorphism rather than only an isomorphism of quotient division rings.

The second assertion follows by applying the same argument successively
inside the indicated compatible localizations.  Then \textup{(4)}
ensures that the final localized algebra, rather than merely a further
localization of it, again belongs to $\mathscr A$.

For the last assertion, $A_n$ is simple and has center $\C$ by
Theorem \ref{thm:minimal-simple} and Corollary \ref{cor:center-solvable}. Lemma~\ref{lem:fraction-center} gives
$Z({\rm Frac}(A_n))=\C$. Under the resulting division-ring isomorphism,
$Z(A(\Lambda,\eta;V_n,\ell_n,\omega_n))$ embeds in this scalar
field. It is therefore $\C$, and Lemma \ref{lem:center} gives
$K_V=0=K_\Lambda$. The total rank-plus-dimension requirement follows
from the tower construction; if an additive coordinate is replaced by a
lattice coordinate, the sum is unchanged.
\end{proof}

\begin{example}\label{ex:jordanian-multiplier}
Not every stage derivation is a scalar-valued canonical derivation
modulo inner derivations. Indeed, fix an integer $d\neq0$ and a
parameter $\xi\in\mathbb C^\times$. Let 
\[
G_d:=\mathbb G_m\ltimes\mathbb G_a,
\qquad
(t,c)(t',c'):=(tt',c+t^dc'),
\]
and let $J_\xi$ be induced by
\[
\mathcal F_\xi
:=
\exp\bigl(-(H/d)\otimes\log(1-\xi E)\bigr),
\qquad
H:=t\frac{\partial}{\partial t},
\qquad
E:=t^d\frac{\partial}{\partial c}.
\]
Then $[t,c]_{J_\xi}=\frac{\xi}{d}t^{d+1}$,
and hence
\[
\mathscr O(G_d)_{J_\xi}
=
\mathbb C[t^{\pm1}][c;\delta],
\qquad
\delta(t)=-\frac{\xi}{d}t^{d+1}.
\]
Since the preceding quotient field is $\mathbb C(t)$, its inner
derivations vanish.  Since a canonical toral derivation satisfies
$\partial_\alpha(t)=\alpha(1)t$, $\delta$ is not canonical modulo
inner derivations.  Rather, we have
\[
\delta=t^d\partial_\alpha,
\qquad
\alpha(1)=-\frac{\xi}{d},
\]
so for $P:=-d\xi^{-1}t^{-d}c$, we get $[P,t]=t$.  Thus, a central homogeneous multiplier is 
necessary. \qed
\end{example}

\subsection{Derivations modulo inner derivations}
\label{sec:normal-family-derivations}

For $A:=A(\Lambda,\eta;V,\ell,\omega)$, retain the notation
$R$, $R_0$, and $\Lambda_0$ of Lemma~\ref{lem:center}, and the
center-valued canonical derivations from
\S\ref{sec:prelim-quantum-differential}. The auxiliary results used below
are collected in \S\ref{sec:prelim-polynomial-cohomology}.

\begin{theorem}\label{thm:center-valued-outer}
Every derivation of $A:=A(\Lambda,\eta;V,\ell,\omega)$ has the form
$\delta=\operatorname{ad}(a)+\partial_{f,g}$,
where
\[
a\in A,
\qquad
f\in\operatorname{Hom}_{\mathbb Z}(\Lambda,Z(A)),
\qquad
g\in\operatorname{Hom}_{\mathbb C}(V,Z(A)).
\]
Equivalently,
\[
\operatorname{Der}_{\mathbb C}(A)
=
\operatorname{InnDer}(A)
+
Z(A)\cdot
\{\partial_{\alpha,\beta}\}.
\]
If the data are jointly nondegenerate, then $Z(A)=\mathbb C$ and
hence $\delta=\operatorname{ad}(a)+\partial_{\alpha,\beta}$.

Moreover, the same statements hold in $\operatorname{Frac}(A)$ for derivations
which preserve $A$.
\end{theorem}

\begin{proof}
Let
\[
B_0
:=
T(V)\Big/
\bigl\langle
P_vP_w-P_wP_v-\omega(v,w)1
\mid\ v,w\in V
\bigr\rangle
\subseteq A.
\]
For every $\mu\in\Lambda$, let
$\tau_\mu\in\operatorname{Aut}_{\mathbb C}(B_0)$ be the translation
automorphism defined by
\[
\tau_\mu(P_v)
:=
P_v-\ell(v)(\mu)1
\qquad(v\in V).
\]
By the defining relations of $A$, we have
\[
X^\mu b=\tau_\mu(b)X^\mu,
\qquad
\tau_\mu(b)=X^\mu b(X^\mu)^{-1}
\qquad(b\in B_0).
\]
Moreover, $\tau_\mu-\operatorname{id}$ strictly lowers the
filtration by total degree in the generators $P_v$, and hence
is locally nilpotent on $B_0$. Thus, for every
$q\in\mathbb C\setminus\{1\}$, the operator
$q\operatorname{id}-\tau_\mu$ is invertible, with inverse
\[
(q\operatorname{id}-\tau_\mu)^{-1}
=
\sum_{k\geq0}
\frac{(\tau_\mu-\operatorname{id})^k}{(q-1)^{k+1}},
\]
the sum being finite on each element of $B_0$.

Now give $A$ its $\Lambda$-grading
$\deg X^\lambda:=\lambda$ and $\deg P_v:=0$.
Then every derivation is a finite sum of homogeneous derivations.

Let $\delta$ be homogeneous of degree $d$.  If there is
$\mu\in\Lambda$ with $q_\eta(d,\mu)\neq1$, then
$q_\eta(d,\mu)\operatorname{id}-\tau_\mu$ is invertible on the
degenerate Weyl algebra generated by the $P_v$ since
$\tau_\mu-\operatorname{id}$ is locally nilpotent.  Thus, subtracting a
suitable $\operatorname{ad}(aX^d)$ kills $\delta(X^\mu)$.  The
corrected derivation then commutes with $\operatorname{Ad}(X^\mu)$ but it
maps generalized eigenspaces with eigenvalue
$q_\eta(\mu,\lambda)$ to generalized eigenspaces with the distinct
eigenvalue $q_\eta(\mu,\lambda+d)$.  Hence, it is zero.

If $q_\eta(d,\Lambda)=1$ but $\ell(v_0)(d)\neq0$ for some $v_0$, then the
operator
$a\longmapsto[a,P_{v_0}]-\ell(v_0)(d)a$
is invertible since $\operatorname{ad}(P_{v_0})$ is locally
nilpotent.  Thus, after an inner correction, we have
$\delta(P_{v_0})=0$.  The derivation then commutes with
$\operatorname{ad}(P_{v_0})$, while it shifts its generalized
eigenvalue from $\ell(v_0)(\lambda)$ to
$\ell(v_0)(\lambda+d)$.  Thus again, it must vanish.

Thus, every homogeneous derivation of degree
$d\notin\Lambda_0$ is inner.  If $d\in\Lambda_0$, then $X^d$ is a
central unit, and $(X^d)^{-1}\delta$ has degree zero. So,
it remains to treat a degree-zero derivation.  Choose a decomposition $V=S\oplus R$ with $\omega|_S$ nondegenerate.  Then
\[
B_0\cong \mathcal A_m(\mathbb C)\otimes\operatorname{Sym}(R).
\]
By Lemma \ref{lem:weyl}, after subtracting an inner derivation induced by an
element of $B_0$, we may assume that $\delta|_{B_0}$ is the extension,
by zero on the Weyl generators corresponding to $S$, of a derivation
$\theta$ of $\operatorname{Sym}(R)$. Write
$\delta(X^\lambda)=c_\lambda X^\lambda$, $c_\lambda\in B_0$.
Applying $\delta$ to
$X^\lambda b(X^\lambda)^{-1}=\tau_\lambda(b)$ gives
\[
 \theta\bigl(\tau_\lambda(b)\bigr)
 -\tau_\lambda\bigl(\theta(b)\bigr)
 =[c_\lambda,\tau_\lambda(b)]
 \qquad(b\in B_0),
\]
where $\theta$ also denotes its extension to $B_0$. For
$b\in Z(B_0)=\operatorname{Sym}(R)$, the right-hand side is zero. Thus $\theta$
commutes with all the translations $\tau_\lambda$ on $\operatorname{Sym}(R)$.
Its coefficients are therefore invariant polynomials, and hence belong
to $\operatorname{Sym}(R_0)$. We may write
\[
 \theta=\sum_j f_j\partial_{\beta_j},
 \qquad f_j\in\operatorname{Sym}(R_0),\quad\beta_j\in R^*.
\]
Extend each $\beta_j$ to $V^*$ by declaring it zero on $S$.
Since $\operatorname{Sym}(R_0)\subseteq Z(A)$, subtracting
$\sum_j f_j\partial_{0,\beta_j}$ leaves a derivation whose
restriction to $B_0$ is zero.

Now $\delta|_{B_0}=0$, the same conjugation identity implies
$c_\lambda\in Z(B_0)=\operatorname{Sym}(R)$, and the relation
$X^\lambda X^\mu=\eta(\lambda,\mu)X^{\lambda+\mu}$ gives
$c_{\lambda+\mu}=c_\lambda+\tau_\lambda(c_\mu)$.
Lemma \ref{lem:lattice} gives
\[
 c_\lambda=b-\tau_\lambda(b)+\sum_jf_j\alpha_j(\lambda),
\]
with $b\in\operatorname{Sym}(R)$, $f_j\in\operatorname{Sym}(R_0)$, and
$\alpha_j\in\Hom_{\mathbb Z}(\Lambda,\C)$.
Thus subtracting $\operatorname{ad}(b)$ and
$\sum_jf_j\partial_{\alpha_j,0}$ kills the derivation.
This proves the degree-zero statement.  

Finally, multiply the degree-zero decomposition by the central
unit $X^d$ and sum over the finitely many homogeneous components.
This proves the assertion for $A$. If a derivation of
$\operatorname{Frac}(A)$ preserves $A$, apply this decomposition
to its restriction. The identity
$\delta(b^{-1})=-b^{-1}\delta(b)b^{-1}$ $(0\ne b\in A)$
then extends the equality uniquely to $\operatorname{Frac}(A)$.
\end{proof}

\begin{corollary}\label{cor:correct-hopf-stage}
At a Hopf $2$-cocycle Ore stage $A_i=A_{i-1}[z_i;\delta_i]$,
assume that $A_{i-1}$ has already been identified with an
algebra $A_{i-1}=A(\Lambda_{i-1},\eta_{i-1};V_{i-1},\ell_{i-1},\omega_{i-1})$ in $\mathscr A$.  Then
\[
[\delta_i]
\in
Z(A_{i-1})\cdot
\{[\partial_{\alpha,\beta}]\}
\]
modulo inner derivations.  If $A_{i-1}$ is jointly nondegenerate, then $[\delta_i]=[\partial_{\alpha_i,\beta_i}]$.

If more generally,
$\delta_i
=
\operatorname{ad}(a_i)
+
u_i\partial_{\alpha_i,\beta_i}
$
with $u_i$ a central unit in the preceding quotient division ring,
then $P_i:=u_i^{-1}(z_i-a_i)$
satisfies
$[P_i,x]=\partial_{\alpha_i,\beta_i}(x)$. \qed
\end{corollary}

Theorem~\ref{thm:deletion-criterion}, Example~\ref{ex:jordanian-multiplier}, Theorem~\ref{thm:center-valued-outer}, and
Corollary~\ref{cor:correct-hopf-stage} motivate the following.

\begin{conjecture}
\label{conj:stagewise}
Let $G$ be connected solvable and let $J$ be
minimal. 
At a stage $B_i=\Os(G_i)_{J_i}$, choose a support representative of
$J_i$ and its right support coaction as in \S\ref{sec:prelim-support-charts}. Semi-invariants at
that stage refer to this coaction. Under a chosen algebra isomorphism,
this condition is transported along that isomorphism; any localization
used subsequently is required to be compatible with the coaction.
Normality and regularity are required in the coefficient algebra at
the stage where an element is inverted.

There exists a normal additive series and a compatible choice of
coordinates such that 
there are finitely many semi-invariants,
normal in the relevant successive coefficient algebras, whose powers
form compatible denominator sets, such that after inverting them
stage by stage and making triangular changes of the Ore generators,
the resulting localized tower is a jointly nondegenerate algebra
$A(\Lambda,\eta;V,\ell,\omega)$. Here triangular changes are allowed to have an invertible
coefficient in the preceding localized algebra; in particular,
they include $z_i\mapsto u_i^{-1}(z_i-a_i)$ when $u_i$ is a
central unit there. \qed
\end{conjecture}

\begin{question}
\label{q:adapted-rank-one}
Can one choose the normal additive series and make compatible
triangular changes of the preceding variables, so that at every
degenerate stage,
$[\delta_i]=u_i[\partial_{\alpha_i,\beta_i}]$
modulo inner derivations, where \(u_i\) is
a central rational semi-invariant unit? \qed
\end{question}

\begin{remark}
A positive answer to Question \ref{q:adapted-rank-one} gives, after identifying the
preceding quotient division rings,
$\delta_i=\operatorname{ad}(a_i)
+u_i\partial_{\alpha_i,\beta_i}$.
Suppose that the compatible coefficient localizations required
in Corollary \ref{cor:integral-deletion} can also be chosen to contain
$a_i$, $u_i$, and $u_i^{-1}$, with $u_i$ central, and that the
remaining compatibility and presentation hypotheses there hold.
Then the triangular change
$P_i:=u_i^{-1}(z_i-a_i)$
satisfies
$[P_i,b]=\partial_{\alpha_i,\beta_i}(b)$
for every element $b$ of the preceding coefficient algebra,
by Corollary~\ref{cor:correct-hopf-stage}. Iterating the construction of Theorem~\ref{thm:deletion-criterion}
inside these localizations gives Conjecture~\ref{conj:stagewise}. Thus the
finite-localization conclusion requires denominator control
for the multipliers as well as for the inner corrections. \qed
\end{remark}

\subsection{The case \texorpdfstring{$\dim R_u(G)=1$}{dim Ru(G)=1}}

The case of a one-dimensional unipotent radical can be settled completely.

The homogeneous-derivation calculation needed for the proof is
Lemma~\ref{lem:homogeneous-derivations}.

\begin{theorem}
\label{thm:rank-one}
Let $G$ be a connected solvable affine algebraic group with
$\dim R_u(G)=1$
and let $J$ be a minimal Hopf $2$-cocycle for $\Os(G)$.  Then Conjecture~\ref{conj:strong-normal-form} holds; namely, there are jointly nondegenerate data such
that
\[
   \Os(G)_J
   \cong
   A(\Lambda,\eta;\C v,\ell,0)
\]
as algebras. More explicitly, writing
$G=T\ltimes\Ga$
and letting $d\in\Lambda$ be the weight of the action of $T$ on
$\Ga$, we have
$d\in{\rm rad}(q_\eta)$
and there is
$\alpha\in\Hom_{\mathbb Z}(\Lambda,\C)$, such that
\begin{equation}\label{eq:rank-one-normal-form}
   \Os(G)_J
   \cong
   \C^\eta[\Lambda][P;\partial_\alpha],
   \qquad
   [P,X^\lambda]=\alpha(\lambda)X^\lambda.
\end{equation}
Also, the joint nondegeneracy conditions are
\begin{equation}\label{eq:rank-one-nondeg}
   \alpha\neq0,
   \qquad
   {\rm rad}(q_\eta)\cap\ker\alpha=0.
\end{equation}
\end{theorem}

\begin{proof}
Choose a coordinate $u$ on the normal subgroup $\Ga$, so that
$\Os(G)=\C[\Lambda][u]$
and
\begin{equation}\label{eq:rank-one-coproduct}
   \Delta(u)=u\otimes1+X^d\otimes u.
\end{equation}
The restriction of $J$ to $\Os(T)$ gives a $2$-cocycle $\eta$ and hence the
subalgebra $\C^\eta[\Lambda]\subseteq\Os(G)_J$.
Then the Ore presentation has the form
$\Os(G)_J=\C^\eta[\Lambda][u;\delta]$.
Since for $\lambda\in\Lambda$, 
\eqref{eq:rank-one-coproduct} and the definition of the twisted product
give
\begin{align*}
   \delta(X^\lambda)
   &=[u,X^\lambda]_J=
   \bigl(J(u,X^\lambda)-J(X^\lambda,u)\bigr)
     X^{\lambda+d}\qquad (\lambda\in \Lambda).
\end{align*}
Thus, $\delta$ is homogeneous of degree $d$.

If $d\notin{\rm rad}(q_\eta)$, Lemma~\ref{lem:homogeneous-derivations} gives
$\delta=\ad(a_d)$ for a homogeneous element $a_d\in\C^\eta[\Lambda]_d$.  Replacing $u$ by
$u':=u-a_d$ then gives
$[u',\C^\eta[\Lambda]]=0$.
Hence, $\Os(G)_J\cong\C^\eta[\Lambda][u']$
has the nonzero proper ideal $(u')$, contradicting simplicity and
Theorem~\ref{thm:minimal-simple}.  Therefore,
$d\in{\rm rad}(q_\eta)$.

By Lemma~\ref{lem:homogeneous-derivations}, there are a central
homogeneous unit $u_d\in\C^\eta[\Lambda]_d$ and
$\alpha\in\Hom_{\mathbb Z}(\Lambda,\C)$, such that
$\delta=u_d\partial_\alpha$. 
Set $P:=u_d^{-1}u$.
Then $[P,X^\lambda]
   =u_d^{-1}\delta(X^\lambda)
   =\alpha(\lambda)X^\lambda$, which proves \eqref{eq:rank-one-normal-form}.

Finally, by Theorem 
\ref{thm:abelian-realization}, 
$A(\Lambda,\eta;\C v,\ell,0)$ is simple precisely when
$\ell(v)\neq0$ and ${\rm rad}(q_\eta)\cap\ker\ell(v)=0$.
Since $J$ is minimal, $\Os(G)_J$ is simple and
\eqref{eq:rank-one-nondeg} follows with $\ell(v)=\alpha$.
\end{proof}

\begin{corollary}\label{cor:rank-one-conjectures}
Conjectures~\ref{conj:strong-normal-form} and \ref{conj:stagewise} hold whenever
\(\dim R_u(G)\leq1\).
\end{corollary}

\begin{proof}
For $\dim R_u(G)=0$, the group is a torus and the twisted algebra is a
quantum torus.  The case $\dim R_u(G)=1$ is
Theorem~\ref{thm:rank-one}.
\end{proof}

\end{document}